\documentclass{amsart}

\usepackage{amssymb,amsmath, amscd, mathrsfs, yfonts, bbm, graphics, dsfont}

\title[A combinatorial result on asymptotic independence] {A combinatorial result on asymptotic independence relations for random matrices with non-commutative entries}

\author{Mihai Popa and Zhiwei Hao}

\address{Department of Mathematics, University of Texas at San Antonio, One UTSA Circle
San Antonio, Texas 78249, USA, and}
\address{``Simon Stoilow'' Institute of Mathematics of the Romanian Academy, P.O. Box 1-764, 014700 Bucharest, Romania}
\email{mihai.popa@utsa.edu}

\address{School of Mathematics and Statistics, Central South University, Changsha
410083, China}
\email{zw\underline{ }hao8@163.com}

\thanks{This work was partially supported by the NNSFC Cooperation Grant No. 201506370093.}

\newtheorem{claim}{}[section]
\newtheorem{defn}[claim]{Definition}
\newtheorem{thm}[claim]{Theorem}
\newtheorem{lemma}[claim]{Lemma}
\newtheorem{remark}[claim]{Remark}

\newtheorem{cor}[claim]{Corollary}

\newcommand{\Tr}{\textrm{Tr}}

\newcommand{\tr}{\textrm{tr}}

\newcommand{\cA}{\mathcal{A}}

\newcommand{\oi}{\overrightarrow{i}}

\input xy
\xyoption{all}

\usepackage{graphicx}

\begin{document}

\begin{abstract}
The paper gives a general condition under which a semicircular matrix is free independent, or asymptotically free independent from the semicircular matrix obtained by permuting its entries.
In particular, it is shown that semicircular matrices are asymptotically free from their transposes, a result similar to the case of Gaussian random matrices. There is also an analysis of asymptotic second order relations between semicircular matrices and their transposes, with results not very similar to the commutative (i.e. Gaussian random matrices) framework. The paper also presents an application of the main results to the study of Gaussian random matrices and furthermore it is shown that  the same condition as in the case of semicircular matrices gives Boolean independence, or asymptotic Boolean independence when applied to Bernoulli matrices.

\end{abstract}

\maketitle


\section{Introduction}

A recurrent theme in the study of non-commutative probability is describing analogous results to the classic (commutative) framework (see, for example, \cite{voi2},  \cite{ns}, \cite{ep}, or  \cite{ber-pata}, \cite{thorb}). One of the basic results in this topic is that semicircular, respectively Bernoulli distributed elements are the analogue (in the Central Limit Theorem sense) of the Gaussian random variables in the framework of free (see \cite{ns}, \cite{vdn}), respectively Boolean independence (see \cite{sw}, \cite{bpv}).
 The connection between Gaussian and semicircular distributed variables is even stronger, since Gaussian random matrices with independent entries are asymptotically free and semicircular distributed. As detailed below, this paper addresses mainly results concerning asymptotics of matrices with entries in a non-commutative algebra, topic that was investigated since the 1990's (see, for example \cite{r}, \cite{shkl}).

 A particular case of the results presented in \cite{mp2}  is that a Gaussian random matrix is asymptotically  free from its transpose.
 The present work is describing some generalizations of this result in the framework of non-commutative probability relations. More precisely, it describes a class of permutations of the entries of a square matrix (the matrix transpose or the partial transpose from \cite{collins-nechita}, \cite{mingo-popa-part} are just particular cases) with the following property: a semicircular, respectively Bernoulli matrix is  (asymptotically) free, respectively (asymptotically) Boolean independent from the matrix obtained by permuting its entries. There is also a brief application of the results to the study of Gaussian random matrices and a detailed investigation of the second order relations between a semicircular matrix and its transpose, in the spirit of \cite{mp1}, \cite{mp2} and \cite{jp}, although the results are significantly different (see the more detailed description below). The methods employed are mostly combinatorial, heavily relying on the properties of permutations and non-crossing partitions.

 Besides the Introduction, the paper is organized in 5 sections, as follows.

 Section 2 presents some needed definitions and preliminary results in combinatorics and free and Boolean independence.

  Section 3 is devoted to the main results, namely Theorem  \ref{thm:3.1}. More precisely, if
$ \sigma $ is a permutation on the set $ \{ (i, j): 1 \leq i, j \leq N \} $ and $ A $ is a $ N \times N $ matrix with entries in some algebra $ \mathcal{A} $, we denote by $ A^{\sigma } $ the matrix with the $ (i, j)$-entry equal to the $ \sigma (i, j) $ entry of $ A $. If  $ S _N $ is an $ N \times N $ semicircular matrix and $ \sigma $, as above, commutes with the transpose and satisfies
\[
 \sharp\{i\in[N]:\;\text{there exist} \; j, k\in[N] \;\text{such that}\; \;\sigma(i,j)=(k,j)\}=0,
\]
then $S $ and $ S ^\sigma $ are free, while if
\[
\lim_{ N \longrightarrow 0 }
\frac{\sharp\big\{(i,j,k):\sigma(i,j)=(i,k),\;\text{for any}\;i,j,k\in [N]\big\}}{N^2} = 0,
\]
then $ S $ and $ S^\sigma $ are asymptotically (as $ N \longrightarrow \infty $) free.

Section 4 applies the main theorem to obtain some asymptotic freeness results concerning Gaussian random matrices, in particular showing that a Gaussian random matrix is asymptotically free from its left-partial transpose (see also \cite{mingo-popa-part}).

Section 5 presents a Boolean independence version of the main theorem.

The last part of the paper, Section 6, describes the second order independence relations between a semicircular matrix and its transpose, in the spirit of \cite{mp1}, \cite{mp2} and \cite{jp}. We mention that the main results of this last section (see Theorem \ref{thm:2c}) are still different in nature from both the formulas in the commutative framework (see \cite{mp2}) and the case of ensembles free semicircular random matrices (see \cite{jp}).


\section{Preliminaries}


\subsection{Noncrossing and interval partitions}\label{haar}

For $n$ a positive integer   we will denote by $ [ n ] $ the ordered set $ \{ 1, 2, \dots, n \} $. By a \emph{partition} on $ [ n ] $ we will understand a collection  $\pi=\{V_1,V_2,\cdots V_{r(\pi)}\}$ of mutually disjoint subsets of $ [ n ] $ such that $ V_1 \cup V_2 \cup \dots \cup V_{r(\pi)} = [ n ] $. We will use the notations $ P(n) $, respectively $ P_2(n) $ for the set of all partitions, respectively pair-partitions, i.e. partitions such that each of their blocks has exactly 2 elements, on $[n]$. The set $ P(n) $ is a lattice with respect to the refinement order (i.e. $ \sigma $ is less that $ \eta $ if each block of $ \sigma $ is included in some block of $ \eta $).

 A partition $ \pi $ on is said to be \emph{non-crossing} if whenever $1\leq a < b <
 c <d \leq n$ are such that $ a, c $, respectively
 $ b, d $ are in the same block of $ \pi $, then $ a, b, c, d $ are all in the same block of $ \pi $.
 The sets of all non-crossing partitions, respectively non-crossing pair-partitions of
 $[n]$ will be denoted by $NC(n)$, respectively $NC_2(n)$. If $ n $ is odd, then, by convention, $ NC_2 (n ) = \emptyset $.

 The terminology corresponds to the fact that a non-crossing partition
admits a planar representation (linear and circular), i.e., in which the arcs or chords
intersect only at elements of $[n]$.
We can put the points
$1, \cdot\cdot\cdot,n$ on the line (or circle), and connects each point with the next member of its part
 by a fold line (or an internal path). Then, the partition is non-crossing if this can
be achieved without fold line (or arcs) crossing each other. See the Fig. 1 below.
\begin{center}
\includegraphics[height=45mm]{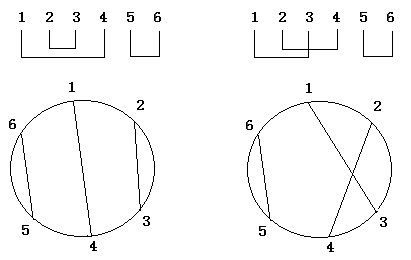}
\end{center}
\begin{center} Fig. 1 Non-crossing (left, ($1,4$), ($2,3$), ($5,6$)) and crossing (right, ($1,5$), ($2,3$), ($4,6$))
partitions on the set $[6]$.
\end{center}

 In the next sections, we will use the following results concerning non-crossing partitions (see \cite{mn}, respectively \cite{ns}, Lecture 9):

 \begin{lemma}\label{lemma:g}
  Let $ n $ be an even positive integer and $ \gamma $ be the permutation on $[ n ] $ with a single cycle, $ (1, 2, \dots, n) $ and $ \sigma \in P_2(n) $. We will identify $ \sigma $ with a permutation by letting $ \sigma(k) = l $ whenever $ (k, l) \in \sigma $. Then there exist a positive integer $ g(\sigma) $ such that
  \[
  \sharp (\gamma \sigma) = \frac{1}{n} + 1 - 2 g(\sigma)
  \]
  Moreover, $ g(\sigma) = 0 $ if and only if $ \sigma $ is non-crossing.
 \end{lemma}

 \begin{lemma} \label{lem:3.2}
 If $ n $ is a positive integer and $ \pi \in NC(n) $, then there exists at least a block of $ \pi $ whose elements are consecutive numbers. In particular, if $ n $ even and $\pi\in NC_2(n)$, there exists some $k\in[n-1]$ such that  $(k,k+1)\in\pi$.
 \end{lemma}

 A non-crossing partition of $ [ n ] $ is called \emph{interval partition} if each of its blocks contains only consecutive elements from $[n]$. We will denote the set of all interval partitions,
respectively pairings of $[n]$ by $\mathcal{I}(n)$, respectively $\mathcal{I}_2(n)$ (if $n$ is odd, then
$\mathcal{I}_2(n) =  \varnothing$). If $n$ is even then $\mathcal{I}_2(n)$ has only one element, namely the partition of
blocks $\{(2k-1, 2k): 1 \leq k \leq n/2\}$.

\subsection{Non-commutative probability spaces and independence relations}\label{section:22}

 Throughout the paper, by a non-commutative probability space we will understand a pair
 $(\mathcal{A}, \varphi)$, where $ \mathcal A $ is a complex, unital $ \ast $-algebra and $ \varphi:\mathcal{A} \longrightarrow \mathbb{C} $ is a unital, linear, positive map.

 For $ r $ a positive integer we define the $ r $-th free, respectively boolean cumulants associated to $ \varphi $ as multilinear  complex maps from $ \mathcal{A}^r $  given by the recurrences:
 \begin{eqnarray}
 \varphi(a_1 a_2\cdot\cdot\cdot a_n) & = &
 \sum_{\pi\in NC(n)}
 \prod_{\substack{V\in \pi\\V=\{v_1,v_2,\cdot\cdot\cdot, v_r\}}}\kappa_r(a_{v_1}, a_{v_2},\cdot\cdot\cdot, a_{v_r})\label{eq:freecum}\\
 & = &
 \sum_{\pi\in \mathcal{I}(n)}
 \prod_{\substack{V\in \pi\\V=\{v_1,v_2,\cdot\cdot\cdot, v_s\}}}\mathfrak{b}_s(a_{v_1}, a_{v_2},\cdot\cdot\cdot, a_{v_s}).\label{eq:boolcum}
 \end{eqnarray}
  For $ \pi \in NC(n) $, we will also use the shorthand notations
   \[
   \kappa_\pi[a_1, a_2,\cdot\cdot\cdot, a_n]
   =\prod_{\substack{V\in \pi\\V=\{v_1,v_2,\cdot\cdot\cdot, v_r\}}}\kappa_r(a_{v_1}, a_{v_2},\cdot\cdot\cdot, a_{v_r})
   \]
   and $ \mathfrak{b}_\pi[a_1, a_2,\cdot\cdot\cdot, a_n ]$ defined analogously.

  A collection of unital  $\ast$-subalgebras (respectively just $ \ast$-subalgebras) $\{ \mathcal{A}_j\}_{ 1 \leq j \leq n} $ of $\mathcal{A}$ are free independent (respectively Boolean) independent  with respect to $ \varphi$ if
 $\varphi(a_1a_2\cdot\cdot\cdot a_m)=0$ (or $\varphi(\overline{a}_1\overline{a}_2\cdot\cdot\cdot \overline{a}_m)=\varphi(\overline{a}_1)\varphi(\overline{a}_2)\cdot\cdot\cdot\varphi(\overline{a}_m)$), for any $m\geq 1$
 whenever $a_j$ (or $\overline{a}_j$), $1\leq j \leq m$, satisfy $\varphi(a_j) = 0$, $a_j \in \mathcal{A}_i(j)$ (or $\overline{a}_j \in\mathcal{A}_i(j)$) with $i(j) \in \{1, 2\}$,
 $i(j) \neq i(j+1)$. An equivalent condition (see \cite{ns}, \cite{sw}) for free, respectively Boolean independence is
 $$ \kappa_m(\overline{a}_1, \overline{a}_2,\cdot\cdot\cdot, \overline{a}_m) = 0 \quad (\text{respectively }\; \mathfrak{b}_m(\overline{a}_1\overline{a}_2\cdot\cdot\cdot \overline{a}_m)=0)$$
 whenever $\overline{a}_j \in\mathcal{A}_{\epsilon(j)}$ such that not all $\epsilon(j)$ are equal. (I. e. mixed free, respectively Boolean cumulants vanish).


  The central limit distributions (see \cite{bpv}) corresponding to free, respectively Boolean independence are the semicircular, respectively Bernoulli distributions. A selfadjoint element $ X \in \mathcal{A} $ is  said to be \emph{semicircular}, respectively   \emph{Bernoulli} distributed  of mean $ 0 $ and variance $ \gamma > 0 $ with respect to $ \varphi $ if
  \[ \varphi(X^n)= \left\{
 \begin{array}{l }
 0 , \ \ \textrm{if $ n $ = odd }\\
  \displaystyle \frac{1}{ m + 1} { 2m \choose m } \gamma^2 = \sharp NC(2m) \cdot \gamma^{2m}\ \  \ \textrm{if $n  = 2m $, even }
 \end{array}
 \right.
 \]
respectively if
\[ \varphi(X^n)= \left\{
 \begin{array}{l }
 0 , \ \ \textrm{if $ n $ = odd }\\
   \gamma^{ 2m },\ \  \ \textrm{if $n  = 2m $, even. }
 \end{array}
 \right.
 \]
  The following result (see \cite{jp} and Theorem 22.3 from \cite{ns}) is a non-commutative version of the Wick Formula (see \cite{jan}) and  will be utilized in the next sections:
  \begin{thm}\label{thm:wick}
  (The free/Boolean Wick Formula)
  \begin{enumerate}
  \item[(i)] If $ \{ s_i\}_{ 1 \leq i \leq m} $ is a family of free \emph{semicircular} non-commutative random variables of mean zero,  and $ x_j $ is a complex linear combination of $ s_i $ for each $ j =1, \dots, n$, then
  \[
    \varphi( x_1 x_2 \cdots x_n)= \sum_{ \pi \in NC_2 (n)} \prod_{ ( i, j) \in \pi } \varphi ( x_i x_j)
    \]
  \item[(ii)] If $ \{ b_i\}_{ 1 \leq i \leq m} $ is a family of boolean independent \emph{Bernoulli distributed} non-commutative random variables,  and $ x_j $ is a complex linear combination of $ b_i $ for each $ j =1, \dots, n$, then
    \[
      \varphi( x_1 x_2 \cdots x_n)=
      \left\{
      \begin{array}{l }
      0,  \ \  \textrm{if $ n $ is odd}\\
      \varphi( x_1 x_2) \varphi( x_2 x_3) \cdots \varphi(x_{ n-1} x_n), \ \ \textrm{if $n $ is even.}
      \end{array}
      \right.
      \]
  \end{enumerate}
  \end{thm}

  We will use the notation $ M_N ( \mathcal{A}) $ for the the algebra of $ N \times N $ matrices with entries from $ \mathcal{A} $,
   i.e. $ M_N( \mathcal{A} )  = M_N ( \mathbb{C} ) \times \mathcal{A} $. The elements of $ M_N ( \mathcal{A} ) $ will be called random matrices with entries in $ \mathcal{A} $. If
   $ A \in M_N (\mathcal{A})$, we will denote the $ (i,j)$ entry of $ A $ by $ [ A]_{ i, j} $.

   The algebra $ M_N( \mathcal{A} ) $ has a non-commutative probability space structure induced by $ ( \mathcal{A} , \varphi ) $. More precisely, the $ \ast$-operation given by $[ A^\ast]_{ i, i} = ( [A]_{i, j})^\ast $   and the unital positive map is given by $ \varphi \circ \tr $, where $ \tr$, respectively $ \Tr $ denote the normalized, respectively non-normalized matrice traces.

      Two sequences $ ( A_N)_{ N \in \mathbb{N} } $ and $ ( B_N )_{ N \in \mathbb{N} } $ such that $ A_N, B_N \in M_N ( \mathcal{A}) $ for each $ N $ are said to be asymptotically free, respectively Boolean, independent if there exist a $ \ast$-probability space $ ( \mathcal{D} , \Phi) $ and two free, respectively Boolean independent non-commutative random variables
       $ a, b \in \mathcal{D} $  such that for any polynomial with complex coefficient $ p (x_1, x_2, x_3, x_4 ) $ in the non-commutating variables $x_1, x_2, x_3, x_4 $ we have that
       \[
       \lim_{ N \longrightarrow \infty } \varphi \circ \tr
        \left(  p ( A_N, A_N^\ast, B_N, B_N^\ast) \right)
        = \Phi( a, a^\ast, b, b^\ast).
       \]



\section{Asymptotic free independence for semicircular matrices}

 \begin{defn}\label{def:31}
 A matrix $ S \in M_N( \mathcal{A} ) $ is said to be semicircular if $ S  = [ c_{ i, j}]_{ i, j =1}^N $ such that
 \begin{enumerate}
 \item[(i)] $ c_{ i, j } = c_{ j, i}^\ast $ for all $ i, j \in [ N] $.
 \item [(ii)] $ \{ \Re c_{i, j}, \Im c_{i, j}: \ 1 \leq i \leq j \leq N  \} $
 is a family of free independent semicircular elements of mean 0 and variance
  $ \frac{ 1}{ \sqrt{2N}} $
  if $ i \neq j $,
  respectively $ \frac{1}{ \sqrt{N} } $
   if $ i = j $.
 \end{enumerate}
 In particular $ \varphi(c_{ij}\cdot c_{lk}) = \delta_{(i, j), (k, l)} $ for any $ i, j, k, l \in [N].$
 \end{defn}

 We will generalize the notion of matrix transpose as follows. Let $ \mathcal{S}([n^2])$ be the set of permutations of $ [ n^2 ] $. If $ A $ is a matrix from $ M_N ( \mathcal{A}) $ and $ \sigma \in \mathcal{S}([n^2]) $, we will denote by $ A^\sigma $ the matrix from $ M_N(\mathcal{A}) $ such that $ [ A^\sigma]_{ i, j } = [A]_{ \sigma(i, j) } $ ( that is, the $i, j$-th entry of $ A^\sigma $ is the $\sigma(i, j) $-th entry of $ A $). In particular, the map  $ t $ given by  $ t(i, j) = (j, i) $ is in $ \mathcal{S}( [ n^2]) $  and $ A^t $ is the matrix transpose of $ A $.

 With the notations from above we shall prove the following theorem.

\begin{thm} \label{thm:3.1}
 \emph{(i)} If $ S $ is a semicircular matrix from $ M_N ( \mathcal{A} ) $ and $ \sigma \in \mathcal{S} ( [ N^2]) $ is such that $ t \circ \sigma = \sigma\circ t $ and
\begin{equation}\label{rel:1}
\{  ( i, j, k) \in [ N ]^3 : \ \ \sigma(i, j) = (i, k) \} = \emptyset
\end{equation}
then $ S $ and $ S^\sigma $ are free with respect to $ \varphi \circ \tr $.

\emph{(ii)} If $ ( S_N ) _{ N \in \mathbb{N}} $ is a sequence of semicircular matrices such that $ S_N \in M_N ( \mathcal{A}) $ and $ ( \sigma( N )  )_{ N \in \mathbb{N}} $ is a sequence of permutations such that, for each $ N $,  $ \sigma(N) $ is an element of $ \mathcal{S}([ N^2]) $ that satisfies  $ \sigma(N) \circ t = t \circ \sigma ( N ) $
and
\begin{equation}\label{rel:2}
\lim_{ N \longrightarrow \infty}
\frac{ \sharp\{ (i, j, k) \in [ N]^3 : \  \sigma_N ( i, j) = (i, k) \}}{N^2} = 0
\end{equation}
 then $ S_N $ and $ { S_N}^{\sigma(N)} $ are asymptotically free.

\end{thm}

\begin{proof}To simplify the notations we will omit the index $ N $ when dealing with the sequences $ ( S_N )_ N $ and $( \sigma(N))_N $ with the convention that only matrices of the same size are multiplied.

  Since $ t \circ \sigma  = \sigma \circ t $, all $ S^\sigma $ are semicircularly distributed of mean 0 and variance 1, hence it suffices to show that mixed free cumulants in $ S $ and $S^\sigma $ vanish, respectively vanish asymptotically (as $ N \longrightarrow \infty $), if $ \sigma $ satisfies property (\ref{rel:1}), respectively property (\ref{rel:2}).

   Let $M $ be a positive integer and $ \overrightarrow{\sigma} = ( \sigma_1, \dots, \sigma_M ) $  with $\sigma_k\in\{1,\sigma\}$ for $1\leq k\leq M$. Denote
  \[
  NC_2(M, \overrightarrow{\sigma}) = \{
  \pi\in NC_2(M) :  \sigma_i = \sigma_j \ \text{for all}\ (i, j) \in \pi \}.
  \]

  Free cumulants of order not 2 in $ S $ or in $ S^\sigma $ are equal to $0$ (since $ S $ and $ S^\sigma $ are semicircular). Henceforth the moment-free cumulant expansion (\ref{eq:freecum})
  gives that the vanishing of mixed free cumulants in $ S $ and  $ S^\sigma $ is equivalent to the equality
  \[
   \sum_{\pi\in NC(M)}
   \prod_{\substack{V\in \pi\\V=\{v_1,v_2,\cdot\cdot\cdot, v_r\}}}\kappa_r(S^{ \sigma_{v_1} }, \cdots, S^{\sigma_{v_r}})
   =
   \sum_{\pi \in NC_2(M, \overrightarrow{\sigma})}
   \prod_{ (i, j) \in \pi } \kappa_2 (S^{ \sigma_i},S^{\sigma_j}).
  \]

  Since $ \kappa_2 ( S, S ) = \kappa_2 ( S^\sigma, S^\sigma) =1 $, we have that
  \[
  \sum_{\pi \in NC_2(M, \overrightarrow{\sigma})}
     \prod_{ (i, j) \in \pi }
      \kappa_2 (S^{ \sigma_i},S^{\sigma_j})
     =
     \sharp NC_2 ( M , \overrightarrow{\sigma})
  \]
  therefore it suffices to show that, for any $ \overrightarrow{\sigma} $,
  \begin{equation}\label{concl}
  \phi(S^{\sigma_1} S^{\sigma_2} \cdots S^{\sigma_M}) -
  \sharp NC_2 ( M , \overrightarrow{\sigma})
   =
   \left\{
   \begin{array}{l l}
   0 & \ \text{if $ \pi $ satisfies property (\ref{rel:1})}\\
   o(1) & \
   \text{if $ \pi $ satisfies property (\ref{rel:2}).}
   \end{array}
   \right.
  \end{equation}

 Denote  $\overrightarrow{i}=(i_1,i_2,\cdot\cdot\cdot,i_M)\in [N]^M$. Then
\begin{align*}\phi\big(S^{\sigma_1}S^{\sigma_2}\cdot\cdot\cdot S^{\sigma_{M}}\big)&=
\text{tr}\circ\varphi\big(S^{\sigma_1}S^{\sigma_2}\cdot\cdot\cdot S^{\sigma_{M}}\big)
\\
&=
\frac{1}{N}
\varphi\Big(\sum_{\overrightarrow{i}\in [N]^M}
[S^{\sigma_1}]_{i_1i_2}[S^{\sigma_2}]_{i_2i_3}\cdot\cdot\cdot [S^{\sigma_M}]_{i_Mi_1}
\Big).
\end{align*}
Applying Theorem \ref{thm:wick}(i), we obtain that $\phi\big(S^{\sigma_1}S^{\sigma_2}\cdot\cdot\cdot S^{\sigma_{M}}\big)$ can be expressed as
\begin{equation*}
\sum_{\pi\in NC_2(M)}\frac{1}{N}
\sum_{\overrightarrow{i}\in [N]^M}
\prod_{\substack{(l,k)\in\pi\\i_{M+1}=i_1}}\varphi\big([S^{\sigma_l}]_{i_li_{l+1}}[S^{\sigma_k}]_{i_ki_{k+1}}\big)
=
\sum_{\pi\in NC_2(M)}v(\pi, \overrightarrow{\sigma}),
\end{equation*}
where we denote $ \displaystyle v(\pi, \overrightarrow{\sigma})=\frac{1}{N}
\sum_{\overrightarrow{i}\in [N]^M} w(\pi, \overrightarrow{i}, \overrightarrow{\sigma}) $
 for
 \[ \displaystyle
 w(\pi, \overrightarrow{\sigma}, \overrightarrow{i})
 =
 \prod_{\substack{(l,k)\in\pi\\i_{M+1}=i_1}}
 \varphi\big([S^{\sigma_l}]_{i_li_{l+1}}
 [S^{\sigma_k}]_{i_ki_{k+1}}\big).
 \]

 We will organize the rest of proof in several steps as follows.

 First, we   shall show that
 \begin{eqnarray}\label{eq:3}
 {}\ \ \ v(\pi, \overrightarrow{\sigma})=\left\{
    \begin{array}{l l}
     1 & \text{if $ \pi \in NC_2 (M, \overrightarrow{\sigma}) $; }\\
    0 &\text{if $ \pi \notin NC_2 (M, \overrightarrow{\sigma}) $   and $ \sigma $ satisfies property (\ref{rel:1});}\\
    o(1) &\text{if $ \pi \notin NC_2 (M, \overrightarrow{\sigma}) $   and $ \sigma $ satisfies property (\ref{rel:2}).}\\
  \end{array}
     \right.
 \end{eqnarray}

  For $ M =2  $,  we have that $ \phi (S^2 ) = \phi ([S^\sigma)]^2) = 1 $, while
  \[ \phi (S\cdot S^\sigma ) = \frac{1}{N}
  \sum_{ i, j =1}^N
  \varphi ( [S]_{ i, j} [S^\sigma]_{ j,i} ) =
  \frac{1}{N^2}
  \sum_{ i, j =1}^N
  \delta_{ (i, j), \sigma(i, j)}
   \]
   so (\ref{eq:3}) is trivial.

   For $ M > 2 $, we will use an inductive argument. For a given $ \pi \in NC_2 (M ) $, according to Lemma \ref{lem:3.2},  there is some $ k $ such that $ (k, k+1)$ is a block of  $\pi $.
 Let
$\pi_{(k)}\in NC_2(M-2)$ be the partition obtained by eliminating the block $(k,k+1)$ from $\pi$ and let
 $ \overrightarrow{\sigma}_{(k)} = ( \sigma_1, \dots, \sigma_{ k-1}, \sigma_{ k+2},\dots,\sigma_M)  $ .
  Then:
\begin{align}
v&(\pi, \overrightarrow{\sigma})
=\frac{1}{N}
\sum_{\overrightarrow{i}}
\prod_{(q,j)\in {\pi}}
\varphi\big([S^{\sigma_q}]_{i_qi_{q+1}}
[S^{\sigma_j}]_{i_ji_{j+1}}\big)\label{eq:v1}\\
=\frac{1}{N} &
\sum_{\overrightarrow{i}\backslash{i_{k+1}}}
[
\prod_{(q,j)\in{\pi}_{ (k)}}
\varphi\big([S^{\sigma_q}]_{i_qi_{q+1}}
[S^{\sigma_j}]_{i_ji_{j+1}}\big)
\cdot
\sum_{i_{k+1}}\varphi\big([S^{\sigma_k}]_{i_ki_{k+1}}[S^{\sigma_{k+1}}]_{i_{k+1}i_{k+2}}\big)
].\nonumber
\end{align}

If $\sigma_k=\sigma_{k+1}$, then
\[
\varphi\big([S^{\sigma_k}]_{i_ki_{k+1}}[S^{\sigma_{k+1}}]_{i_{k+1}i_{k+2}} =
\frac{1}{N}
 \delta_{\sigma_k(i_{k+1}, i_k),\sigma_{k+1}(i_{k+1},i_{k+2})}
 =\frac{1}{N}\delta_{i_k,i_{k+2}},
\]

 \noindent so, with the notation
 $ \overrightarrow{i}_{(k)} = ( i_1, \dots, i_{ k-1}, i_{ k+2}, \dots, i_M) $,  the equation (\ref{eq:v1}) becomes
\begin{align*}
v(\pi, \overrightarrow{\sigma})
&
=\frac{1}{N}
\sum_{\overrightarrow{i}\backslash{i_{k+1}}}
[
\prod_{(q,j)\in{\pi}_{ (k)}}
\varphi\big([S^{\sigma_q}]_{i_qi_{q+1}}
[S^{\sigma_j}]_{i_ji_{j+1}}\big)
\cdot
\sum_{i_{k+1}}\frac{1}{N}\delta_{i_k,i_{k+2}}]\\
&
=\frac{1}{N}
\sum_{\overrightarrow{i}_{(k)}}\prod_{(q,j)\in{\pi}_{ (k)}}
\varphi\big([S^{\sigma_q}]_{i_qi_{q+1}}
[S^{\sigma_j}]_{i_ji_{j+1}}\big)
= v (\pi_{(k)}, \overrightarrow{\sigma}_{(k)} )
\end{align*}
 and property (\ref{eq:3}) follows.

 If $\sigma_k \neq \sigma_{ k + 1 } $, then
  $ \pi \notin NC_2(M, \overrightarrow{\sigma)}).$
    Suppose first that $ \sigma $ satisfies property (i). Then, for any values of $i_k, i_{k+1}, i_{k+2}$, according to equation (\ref{rel:1}), we have that
 $ \sigma_k(i_k, i_{ k + 1})
 \neq \sigma_{k + 1} (i_{ k+1}, i_{ k+ 2}) $
 , hence
 \[
 \varphi\big([S^{\sigma_k}]_{i_ki_{k+1}}[S^{\sigma_{k+1}}]_{i_{k+1}i_{k+2}} =
 \frac{1}{N}
  \delta_{\sigma_k(i_{k+1}, i_k),\sigma_{k+1}(i_{k+1},i_{k+2})}=0
 \]
so equation (\ref{eq:5}) gives $ v(\pi, \overrightarrow{\sigma}) = 0 $, and property (\ref{eq:3}) follows.


  Next, we shall show the following auxiliary property (see also the diagram below):

 \begin{center}
  \includegraphics[height=10mm]{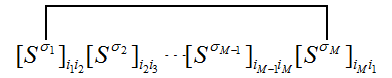}
 \end{center}

  \emph{If $\pi\in NC_2(M)$ with $(1,M)\in\pi$    and if the triple
 $(i_1,i_2,i_M)$ is given, then}
 \begin{equation} \label{eq:5}
 \sharp \Big\{(i_3,i_4,\cdot\cdot\cdot,i_{M-1}):
  w(\pi, \overrightarrow{i})\neq0\Big\}\leq      N^{\frac{M-2}{2}}.
\end{equation}

 We shall prove (\ref{eq:5}) by induction. If $ M =2 $, the property is trivial. For the induction step,
 suppose that $ \pi \setminus ( 1, M) $ has exactly $ r $ exterior blocks  denoted  $   ( q (l ) + 1 , q (l  + 1 ) ) $ with $ 1 \leq l \leq r  $  and $ q (1) = 1, q(r+1) = M -1 $ (see the diagram below).

 \vspace{1.1cm}

  \noindent Note that it suffices (see also the diagram below) to show that
 \begin{equation} \label{eq:6}
 \sharp \Big\{
 (i_{q(1) +1}, i_{q(1) +2}, i_{q(2)},\dots,
 i_{ q(r) +1},
  i_{q(r) +2}, i_{q(r+1)}):
 w(\pi, \overrightarrow{i},\overrightarrow{\sigma} )\neq0
 \Big\}\leq N^r,
 \end{equation}

 \begin{center}
 \includegraphics[height=13.2mm]{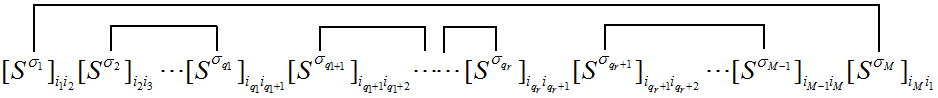}
 \end{center}

\noindent  and (\ref{eq:5}) will follow applying the induction hypothesis to the restrictions of $ \pi $ to each of the the ordered sets $ \{ q(l) +2, q(l) + 3, \dots, q(l+1) - 1 \} $ (i.e. on the elements ``under'' each of the exterior blocks of $ \pi \setminus ( 1, M) $).

  To prove (\ref{eq:6}), remark that
  $w(\pi, \overrightarrow{i})\neq0$
   implies that, for each $ l $
   \[\varphi
  \big(
  [S^{\sigma_{q(l) +1}}]_{i_{q(l)+1}i_{q(l)+2}}
  [S^{\sigma_{q(l+1)}}]_{i_{q(l+1)}i_{q(1+1)+1}}\big)\neq0,
  \]
   that is,
$\sigma_{q(l) +1}(i_{q(l)+1}, i_{q(l)+2})
=\sigma_{q(l+1)}(i_{q(l+1)}, i_{q(1+1)+1})$,
 which is equivalent to $(i_{q(l+1)},i_{q(l+1)+1})=
 \sigma_{q(l)}^{-1}\circ\sigma_{q(l+1)}
 (i_{q(l)+1},i_{q(l)+2})$. Since
$i_2=i_{q(1) +1}$ is fixed, an inductive argument gives that the right hand side of equation (\ref{eq:6}) equals
\[
\sharp\Big\{
 (i_{q(1) +2}, i_{q(2) +2}, \dots, i_{q(r)+2}): \ w(\pi, \overrightarrow{i})\neq0
 \Big\}
 \leq N^r
\]
that is (\ref{eq:6}), hence (\ref{eq:5}) is also proven.

 Remark that (\ref{eq:5}) is equivalent to the following property:

 \emph{Suppose $ \pi \in NC_2 (M ) $ and $ k  \in [ M ] $ is such that $ (k, k+1) \in \pi $ and that  the triple $( i_k, i_{ k+1}, i_{ k + 2}) $ is given. Then:}
 \begin{equation} \label{eq:7}
\sharp \Big\{(i_1,\dots i_{k-1},i_{k+3},\dots,i_{M}):
w(\pi, \overrightarrow{i})\neq0
\Big\}\leq N^{\frac{M-2}{2}}.
 \end{equation}

 To see the equivalence between (\ref{eq:5}) and (\ref{eq:7}), let $ \eta $ be the circular permutation of  $[ M ] $
 given by
 $ \eta(p) = p - k ( \text{mod } M ) $ with the convention $ \eta(k)= M $. Define $ \eta(\pi) $ the pair-partition of $ [M] $ given by  $ (i, j) \in \eta(\pi) $ if and only if $ (\eta^{ -1}(i), \eta^{ -1} (j) ) \in \pi $ and $ \overrightarrow{j} = ( j_1, j_2, \dots, j_M) $ given by
  $ j_p = i_{\eta^{ -1}(p) } $.

  Since $ \eta $ is a circular permutation of the ordered set $[ M]$ and $ \pi $ is non-crossing, it follows that $ \eta(\pi) $ is also non-crossing ($ \pi $ and $ \eta(\pi) $ have the same circle diagram, see the figure below, with the elements from the codomain of $ \eta $ represented by $ \overline{1}, \dots, \overline{M} $).

  \begin{center}
  \includegraphics[height=40mm]{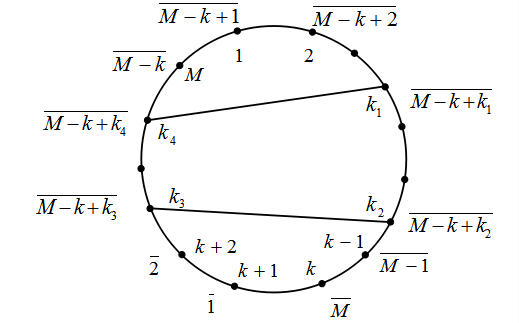}
  \end{center}
  \begin{center} \textbf{Fig.2.} For each $(k_1,k_4)$ and $(k_2,k_3)$ belong to $\pi\in NC_2(M)$, $(\overline{M-k+k_1},\overline{M-k+k_4})$ and $(\overline{M-k+k_2},\overline{M-k+k_3})$
  are non-crossing.
  \end{center}

   Take
   $ \overrightarrow{\omega} = ( \omega_1. \dots, \omega_M ) $
    such that
    $ \omega_k = \sigma_{ \eta^{-1}(k)}$.
     Then the definitions of
     $ \overrightarrow{j} $ and $ \overrightarrow{\omega} $
     give that
   $ w ( \pi, \overrightarrow{i}, \overrightarrow{\sigma}) =
   w ( \eta(\pi), \overrightarrow{j}, \overrightarrow{\omega}) $. Moreover,
      $(j_1, j_2, j_M ) = ( i_{ k}, i_{ k + 1}, i_{ k + 2 }) $ (which is fixed) and $ (k, k+1)\in \pi $ is equivalent to $ (1, M) \in \eta(\pi) $, so property (\ref{eq:7}) for $ \pi $ is equivalent to property (\ref{eq:5}) for $ \eta(\pi) $.

Suppose now that there exists a block $(k, k+1)$ of $ \pi $ such that $ \sigma_k \neq \sigma_{ k+1} $  and that $\sigma $ satisfies property (\ref{rel:2}).

 To simplify the writing, for each triple $(i_k, i_{k+1}, i_{ k + 2})  \in [ N ] ^3 $,
 let
 \[ \mathcal{I}( i_k, i_{ k+1}, i_{ k + 2}) = \{\overrightarrow{i}_{(k)}: \  w(\pi, \overrightarrow{\sigma}, \overrightarrow{i}) \neq 0  \}.
 \]

 Since $ (k, k+1) \in \pi $, we have that
 $ \varphi
 \big([S^{\sigma_k}]_{i_ki_{k+1}}
 [S^{\sigma_{(k+1)}}]_{i_{k+1}i_{k+2}}\big) $
  is a factor in the product
$ w (\pi, \overrightarrow{\sigma}, \overrightarrow{i}) $ and, from Definition \ref{def:31}, its value is either 0 or $ \frac{1}{N} $, therefore
\begin{align*}
v( \pi, \overrightarrow{\sigma}) =&
\frac{1}{N} \sum_{ \overrightarrow{i} }
w (\pi, \overrightarrow{\sigma}, \overrightarrow{i}) \\
&=
\sum_{ \overrightarrow{i}}
[\varphi
 \big([S^{\sigma_k}]_{i_ki_{k+1}}
 [S^{\sigma_{(k+1)}}]_{i_{k+1}i_{k+2}}\big)
 \cdot w (\pi, \overrightarrow{\sigma}, \overrightarrow{i}) ]
\\
 =
\sum_{ i_k, i_{ k+ 1}, i_{k + 2 } = 1}^N
&
[\varphi
 \big([S^{\sigma_k}]_{i_ki_{k+1}}
 [S^{\sigma_{(k+1)}}]_{i_{k+1}i_{k+2}}\big) \cdot
 \sum_{ \overrightarrow{i}_{(k)}\in \mathcal{I}( i_k, i_{ k+1}, i_{ k+2}) }
 w (\pi, \overrightarrow{\sigma}, \overrightarrow{i})].
\end{align*}
$w(\pi, \overrightarrow{\sigma}, \overrightarrow{i}) $ is product of $\frac{M}{2} $ factors of the form
$ \varphi\big([S^{\sigma_t}]_{i_t i_{t+1}}
 [S^{\sigma_{(s)}}]_{i_{s} i_{s+1}}\big) $,each of which equals either $ \frac{1}{N} $ or $0$, so the equation above becomes
 \[
  v(\pi, \overrightarrow{\sigma}) =
  \sum_{ i_k, i_{ k+ 1}, i_{k + 2 } = 1}^N
  [
  \varphi
   \big([S^{\sigma_k}]_{i_ki_{k+1}}
   [S^{\sigma_{(k+1)}}]_{i_{k+1}i_{k+2}}\big)
   \cdot
   N^{\frac{M}{2}} \cdot \sharp \mathcal{I}( i_k, i_{ k+1}, i_{ k+2})
  ]
  \]
  Property (\ref{eq:7}) gives that $ \sharp\mathcal{I}( i_k, i_{ k+1}, i_{ k+2}) \leq N^{\frac{M-2}{2}} $ for each $(i_k, i_{ k+1}, i_{ k+2}) $, therefore
  \begin{align*}
  v(\pi, \overrightarrow{\sigma})
  &\leq
  \sum_{ i_k, i_{ k+ 1}, i_{k + 2 } = 1}^N \frac{1}{N}
   \varphi
     \big([S^{\sigma_k}]_{i_ki_{k+1}}
     [S^{\sigma_{(k+1)}}]_{i_{k+1}i_{k+2}}\big)\\
     &\leq \frac{1}{N^2}
     \sharp\{ (i, j, k) \in [ N]^3 : \  \sigma ( i, j) = (i, k) \}=o(1).
  \end{align*}
 and the proof of (\ref{eq:3}) is complete.

  Remark that
  \begin{align*}
 \phi(S^{\sigma_1} \cdots S^{\sigma_M})
 & = \sum_{ \pi \in NC_2(M)}v( \pi, \overrightarrow{\sigma}) \\
&=
   \sum_{
  \pi \in NC_2( M, \overrightarrow{\sigma})
   }
   v( \pi, \overrightarrow{\sigma})
   +
   \sum_{
     \pi \notin NC_2( M, \overrightarrow{\sigma})
      }
      v( \pi, \overrightarrow{\sigma})
  \end{align*}
 and (\ref{eq:3}) gives
 \[
 \phi(S^{\sigma_1} \cdots S^{\sigma_M})
 =\left\{
 \begin{array}{l l}
 \sharp NC_2(M, \overrightarrow{\sigma}), &
 \text{if $ \pi $ satisfies property (i) }\\
 \sharp NC_2(M, \overrightarrow{\sigma}) + o(1), &
 \text{if $ \pi $ satisfies property (ii)}
 \end{array}
 \right.
 \]
that is (\ref{concl}).

\end{proof}


An immediate consequence of the result above is the following:
\begin{cor}\label{cor:3.3}
Let $S_N=[c_{ij}]_{1\leq i,j\leq N}\in M_N(\mathcal{A})$ be a semicircular matrix,
 then $\{S_N\}$ and $\{S_N^t\}$ are asymptotically free with respect to $\phi:=\text{tr}\circ\varphi$.
\end{cor}

\begin{proof} We have that $S_N^t $ is given by the permutation $ t(i, j) = (j, i)$. Since

\[
 \frac{\sharp \{ (i, j, k): \ t(i,j) = (i, k\}}
 {N^2} = \frac{1}{N}
\]
we get that $ t $ satisfies property (\ref{rel:2}), hence the conclusion.
\end{proof}


\section{ An Application to Gaussian Random Matrices}

 Let $(\Omega, P) $ be a (classical) probability space. We by $ E $ the expectation $ E(\cdot) =\int_{\Omega} \cdot dP $ and by
 $ L^{ - \infty}(\Omega, P) = \cap_ {1 \leq p  < \infty } L^p( \Omega, P)  $. Remark that $ ( L^{ - \infty}(\Omega, P),  E ) $ is a non-commutative probability space in the sense of Section \ref{section:22}.

 \begin{defn}
  By a $ N \times N $  Gaussian random matrix we will understand a selfadjoint element $ G = [ g_{i,j}] _{i, j =1}^{N} $ of
  $ M_{ N \times N }(L^{ - \infty}(\Omega, P) ) $
  such that
   $\{ g_{ i, i}: \ 1\leq i \leq N \}\cup
   \{ \Re g_{i, j}, \Im g_{i, j}: \ 1 \leq i < j \leq N \}
   $
   form an independent family of normally distributed random variables of mean $ 0 $ and
   variance $ \frac{1}{N}$ if $ i = j $, respectively $ \frac{1}{2N} $ if $ i \neq j $ .
 \end{defn}

 The next result is a refinement of Proposition 22.22 and Exercise 22.25 from \cite{ns}.

 \begin{lemma} \label{lem:ref}
 Suppose that $ (\cA, \varphi)$ is a non-commutative probability space, $ m $ is a fixed positive integer and
   $S = [ c_{ij}]_{ i, j =1}^m $
 is a $ m \times m $ semicircular matrix from $ M_m(\cA) $.

 Let
 $ G = [ g_{ ij}]_{ij =1}^{mN} $
  be a $ mN \times mN $ Gaussian random matrix and, for $ k,l \in [ m ] $, let $ G_{ k l } $ be the $ N \times N $ random matrix given by
  \[
    [ G_{ k, l}]_{ i, j} = g_{ km+ i, lm + j }
    \]
 i.e. $ G $ equals the block-matrix
 $ [ G_{k,l}]_{k, l =1}^m $.

 Then, for any
  $ \overrightarrow{j} = (j_1, j_{-1}, \dots, j_{n},j_{-n}) \in [ m]^{2n} $, with the notation from Lemma \ref{lemma:g},
  we have that
 \begin{align*}
  E \circ \tr( G_{j_1, j_{-1} } G_{j_2 j_{-2}}\cdots
  G_{j_n j_{-n}} ) -
  &
  \varphi(c_{j_1, j_{-1} } c_{j_2 j_{-2}}\cdots
    c_{j_n j_{-n}} ) \\
  &= m^{-\frac{n}{2}} \sum_{ \pi \in P_2(n) \setminus NC_2(n) }
  \big(  N^{ -2g(\pi)} \prod_{ k \in [n]} \delta_{ j_k j_{ - \pi(k)}} \big).
  \end{align*}
   where we identify $ \pi \in NC_2(n) $ to a permutation on $ [ n ] $ via $ \pi(k) = l $ whenever $ (k, l) \in \pi $.
 \end{lemma}

 \begin{proof}

  Let $ \oi = (i_1, i_2, \dots, i_n) \in [N]^n $ and denote by $ g^{(j_kj_{-k})}_{i_ti_s} $ the $ (i_t i_s) $-entry of $ G_{j_k j_{-k}} $. With this notation, the Wick Formula gives that
  \begin{align*}
  E \circ \tr&( G_{j_1, j_{-1} } G_{j_2 j_{-2}}\cdots G_{j_n j_{-n}} )
  =
   \sum_{ \oi \in [ N]^n} \frac{1}{N}
 E( g^{(j_1j_{-1})}_{ i_1 i_2} g^{(j_2j_{-2})}_{ i_2 i_3}
 \cdots g^{(j_nj_{-n})}_{i_n i_1})\\
 &=\frac{1}{N} \sum_{ \pi\in P_2(n)} \sum_{ \oi \in [N]^n }
  \prod_{ (k, l) \in \pi}
  E \big( g^{(j_k j_{-k})}_{i_k i_{ k+1}} g^{( j_l j_{l+1})}_{i_l i_{l+1}} \big)\\
  &=
  \frac{1}{N} \sum_{ \pi\in P_2(n)} \sum_{ \oi \in [N]^n }
   \prod_{ (k, l) \in \pi}
   \frac{1}{mN} \delta_{(j_k, j_{-k}), (j_{-l}, j_{l})}\delta_{i_k, i_{l+1}} \delta_{i_l, i_{k+1}}\\
   & = m^{-\frac{n}{2}} N^{ -(\frac{n}{2} +1)}
   \sum_{ \pi\in P_2(n)}[
  \big( \prod_{k \in [n]}
   \delta_{(j_k, j_{-k}), (j_{ -\pi(k)}, j_{\pi(k)})} \big)
   \cdot
  \big(  \sum_{ \oi \in [N]^n}\prod_{k \in [n]}
   \delta_{i_k, i_{\pi\gamma(k)}} \big)
   ]
   .
   \end{align*}
   The product
    $ \displaystyle  \prod_{k \in [n]}  \delta_{i_k, i_{\pi\gamma(k)}} $
     is nonzero if and only if $ i_k = i_l $ whenever $ k, l $ are in the same cycle of $ \pi\gamma $, hence Lemma \ref{lemma:g} gives that
   \[
   \sum_{ \oi \in [N]^n} \prod_{ k \in [n]} \delta_{i_k, i_{ \pi\gamma(k)}}
    = N^{ \sharp (\pi\gamma)} = N^{ 1+ \frac{n}{2} - 2g(\pi)}
    \]
    so, another application of Lemma \ref{lemma:g} gives
    \begin{align*}
    E \circ \tr( G_{j_1, j_{-1} }\cdots G_{j_n j_{-n}} )
   & = m^{- \frac{n}{2}} \sum_{ \pi\in P_2(n)}
    \big(
     N^{-2g(\pi)} \prod_{ k \in [N]}
 \delta_{(j_k, j_{-k}), (j_{\pi(k)}, j_{-\pi(k)})} \big)\\
 \end{align*}
    On the other hand,
    $ \displaystyle \varphi( c_{j_k j_{-k}} c_{j_l j_{-l}}) = \frac{1}{m}
    \delta_{(j_k, j_{-k}), (j_{ - \pi(k)}, j_{\pi(k)})} $, so Theorem \ref{thm:wick} gives that
    \begin{align*}
     \varphi(c_{j_1, j_{-1} } c_{j_2 j_{-2}}\cdots c_{j_n j_{-n}} ) &
     = \sum_{ \pi\in NC_2(n)} \prod_{ (k, l) \in \pi} \varphi( c_{j_k j_{-k}} c_{j_l j_{-l}})\\
     =&
      m^{- \frac{n}{2} } \sum_{\pi\in NC_2(n)} \big(\prod_{ k \in [N]}
       \delta_{(j_k, j_{-k}), ( j_{-\sigma(k)}, j_{\sigma(k)})} \big).
    \end{align*}
      But  $ \delta_{(j_k, j_{-k}), (j_{ - \pi(k)}, j_{\pi(k)})} = \delta_{j_k,j_{ - \pi(k))}} \delta_{j_{ \pi(k)} j_{ - \pi( \pi(k))} } $,
      and, from Lemma \ref{lemma:g},  $ g(\pi) = 0 $ for $ \pi \in NC_2(n) $, hence
      the conclusion.
 \end{proof}

  For $ \sigma $ a permutation in $ \mathcal{S}([m]^2) $ and $ G $ a $ mN \times mN $ Gaussian random matrices seen as a $ m \times m $ block matrix $ G = [ G_{i, j}]_{i, j =1}^m $, we will denote by $ G^{\lceil \sigma \rceil} $ the random matrix  with block entries $ [ G^{\lceil \sigma \rceil}]_{ i, j} = G_{\sigma(i, j)} $. For $ \sigma = t $, the matrix $ G^{\lceil \sigma \rceil} $ is  the left-partial transpose of $ G$ (see \cite{}).

 The following result is an immediate corollary of Lemma \ref{lem:ref} and Theorem \ref{thm:3.1}.

 \begin{thm}
   Let $ G $ be a $ mN \times mN $ Gaussian random matrix. With the notations above, we have that:
   \begin{itemize}
   \item[(i)] If $ m $ is fixed and $ \sigma \in \mathcal{S}([m]^2) $ is such that $ t \circ \sigma = \sigma\circ t $ and
    \[
    \{  ( i, j, k) \in [ m ]^3 : \ \ \sigma(i, j) = (i, k) \} = \emptyset
    \]
    then $ G $ and $ G^{\lceil \sigma \rceil} $ are asymptotically free as $ N \longrightarrow \infty$.
   \item[(ii)] The random matrices $ G $ and $ G^{ \lceil t \rceil} $ are asymptotically free as both $m \longrightarrow \infty $ and $ N \longrightarrow\infty $.
   \end{itemize}
 \end{thm}

 \begin{proof}
  Consider a positive integer $n $ and a mapping $ \varepsilon : [ n ] \longrightarrow \{ (1), (-1)\} $. To simplify the writing, we denote by $G^{(-1)} =G $ and $ G^{(-1)} = G^{ \lceil \sigma \rceil} $.

  Let $ S = [ c_{i, j}]_{i, j =1}^m $ be a $ m \times m $ semicircular matrix and denote $ S^{(1)} = S $, respectively $ S^{(-1)}=S^{[ \sigma]} $.

  Then
  $ \big[ G^{ \lceil \sigma\rceil}\big]_{ i, j} = \big[ G \big] _{ \sigma(i, j)} $ and
   $ \big[ S ^{ [ \sigma ] } \big]_{ i, j} = \big[ S\big] _{ \sigma(i, j)} $,
    so, for
    $ j_{ -k} = j_{ k+1} $ if $ k \in [ n] $ (we identify, as before, $ n + 1 $ to 1),
      Lemma \ref{lem:ref}  gives that
    \begin{align}\label{eq:23}
     E \circ \tr \big(
  \big[ G^{\varepsilon(1)} \big]_{ j_1 j_2}
  \big[ G^{\varepsilon(2)} \big]_{ j_2 j_3}
  &\cdots
   \big[ G^{\varepsilon(n)} \big]_{ j_n j_1}
            \big)
            -
  \varphi \big(
  \big[ S^{\varepsilon(1)} \big]_{ j_1 j_2} \cdots
  \big[ S^{\varepsilon(n)} \big]_{ j_n j_1}
  \big)   \\
  =   m^{ - \frac{n}{2}}
  \sum_{ \pi \in P_2(n) \setminus NC_2(n) } &
  \big(
  N^{ -2g(\pi)} \cdot \prod_{ (k, l) \in \pi }
  \delta_{ \sigma_k( j_{k}, j_{ k+1}), \sigma_{l}( j_{l+1}, j_l) }
  \big).
  \nonumber
     \end{align}
 where
  $ \sigma_k(i, j) = (i, j) $ if $ \varepsilon(k) = (1)$,
   respectively
   $ \sigma_k(i, j) = \sigma(i, j) $
    if $ \varepsilon(k) = (-1) $.

  Suppose first that $ m $ is fixed and $ \sigma $ satisfies the properties from (i). Then, from Theorem \ref{thm:3.1}, $ S $ and $ S^{[ \sigma]} $ are free with respect to $ \varphi \circ \tr$, hence it suffices to show that for ay $ n $ and any mapping $ \varepsilon $ as above,
  \begin{equation}\label{eq:431}
     \lim_{N \longrightarrow\infty} E \circ \tr( G^{\varepsilon(1)} G^{\varepsilon(2)} \cdots G^{\varepsilon(n)}) - \varphi\circ \tr(S^{\varepsilon(1)} \cdots S^{ \varepsilon(n)}) = 0.
    \end{equation}

   But, since $ m $ is fixed and $ g(\pi) \geq 1 $ if $ \pi \in P_2(n) \setminus NC_2(n) $, equation (\ref{eq:23}) gives that
   \[
   E \circ \tr( G^{\varepsilon(1)} G^{\varepsilon(2)} \cdots G^{\varepsilon(n)}) - \varphi\circ \tr(S^{\varepsilon(1)} \cdots S^{ \varepsilon(n)}) = O(N^{-2})
   \]
   thus (\ref{eq:431}) holds true.

   Suppose now that $ \sigma = t $. From Corollary \ref{cor:3.3}, we get that $S $ and $ S^{[\sigma]} $ are asymptotically free as $ m \longrightarrow \infty $, so it suffices to show that
   \begin{equation} \label{eq:432}
     \lim_{ N \longrightarrow \infty} E \circ \tr( G^{\varepsilon(1)} G^{\varepsilon(2)} \cdots G^{\varepsilon(n)}) - \varphi\circ \tr(S^{\varepsilon(1)} \cdots S^{ \varepsilon(n)}) = 0.
    \end{equation}

   For
    $ \overrightarrow{j} = ( j_1, j_2, \dots, j_n) $,
     equation (\ref{eq:23}) gives
  \begin{align}\label{eq:435}
   E &\circ  \tr( G^{\varepsilon(1)}  G^{\varepsilon(2)} \cdots G^{\varepsilon(n)}) - \varphi\circ \tr(S^{\varepsilon(1)} \cdots S^{ \varepsilon(n)})
   \\
 &  = \sum_{ \pi \in P_2(n) \setminus NC_2(n) }
  \big[ N^{ -2g(\pi)} \cdot
   \big(
   m^{ -(\frac{n}{2} + 1)}
   \sum_{ \overrightarrow{j} \in [ m ]^n }
    \prod_{ (k,l) \in \pi}
    \delta_{ \sigma_k( j_{k}, j_{ k+1}), \sigma_{l}( j_{l+1}, j_l) }
   \big)
   \big] \nonumber \\
 &  \leq  N^{-2} \cdot \frac{1}{m}
   \sum_{ \pi \in P_2(n)}
    \big[
    \sum_{ \overrightarrow{j} \in [ m ]^n }
    \prod_{ (k,l) \in \pi}
   \frac{1}{m}
   \delta_{ \sigma_k( j_{k}, j_{ k+1}), \sigma_{l}( j_{l+1}, j_l) }
   \big].\nonumber
  \end{align}

  Now let $ X = X^{(1)} $ be a $ m \times m $ Gaussian random matrix. Denote by $ X^{(-1)} $ the transpose of $ X $ and let $ x^{(1)}_{i, j} $, respectively $ x^{(-1)}_{ i, j} $ be the $(i, j)$-entry of $ X^{(1)} $, respectively of $ X^{(-1)} $. Then, with $ \varepsilon(k) $ and $ \sigma_k $ as defined above, we have that
  \[
  \frac{1}{m}
     \delta_{ \sigma_k( j_{k}, j_{ k+1}), \sigma_{l}( j_{l+1}, j_l) }
     = E(x^{\varepsilon(k)}_{j_k, j_{k +1}} x^{\varepsilon(l)}_{j_l, j_{-l}} ).
  \]
  Thus
  \begin{align*}
   E \circ & \tr( G^{\varepsilon(1)}  G^{\varepsilon(2)} \cdots G^{\varepsilon(n)}) - \varphi\circ \tr(S^{\varepsilon(1)} \cdots S^{ \varepsilon(n)})
     \\
     &\leq N^{-2} \frac{1}{m}
     \sum_{ \overrightarrow{j} \in [m]^n}
     \big[
     \sum_{ \pi \in P_2(n)}
     \prod_{ (k, l) \in \pi }
     E(x^{\varepsilon(k)}_{j_k, j_{k +1}} x^{\varepsilon(l)}_{j_l, j_{-l}} )
     \big] =
     N^{-2} \cdot E\circ \tr \big(
     X^{ \varepsilon(1)} \cdots X^{\varepsilon(n)}\big).
  \end{align*}

   But, as shown in \cite{mp2}, the random matrices $ X $ and $ X^t $ are asymptotically (as $ m \longrightarrow \infty $) free and semicircular distributed, hence the expression
    $ E\circ \tr \big(
 X^{ \varepsilon(1)} \cdots X^{\varepsilon(n)}
 \big) $
 converges as
  $ m \longrightarrow\infty $,
   therefore is bounded for $ m $ positive integer.  It follows that
 \[
 \lim_{ N \longrightarrow \infty } N^{-2}
  E\circ \tr \big(
          X^{ \varepsilon(1)} \cdots X^{\varepsilon(n)}\big)
          = 0.
   \]
   Remark that the right-hand side of the equality in equation (\ref{eq:435}) is positive, henceforth the proof of (\ref{eq:432}) is complete.

   \end{proof}


\section{Asymptotic Boolean independence and Bernoulli matrices}

\begin{defn}
Let $(\mathcal{A}, \varphi)$ be a non-commutative probability space and $ N $ be a positive integer. A $ N \times N $ Bernoulli matrix (over $ \mathcal{A} $) is a matrix $ B = [ b_{i,j}]_{i, j =1}^N $ such that
\begin{itemize}
\item[(i)]$ b_{i, j} =b_{i,j}^\ast$ for all $ i,j \in [ N ] $.
\item[(ii)]$ \{ \Re c_{i, j}, \Im c_{i, j}: \ 1\leq i \leq j \leq N  \} $ is a family of Boolean independent Bernoulli-distributed elements of mean $0 $ and variance $ \frac{1}{\sqrt{2N}} $ if $ i \neq j $, respectively $ \frac{1}{\sqrt{N}} $ if $ i = j $.
\end{itemize}
In particular $ \varphi(c_{ij}\cdot c_{lk}) = \delta_{(i, j), (l, k)} $ for any $ i, j, k, l \in [N].$
\end{defn}

An immediate consequence of Theorem \ref{thm:3.1}(ii)
is that Bernoulli matrices are Bernoulli-distributed with respect to $ \phi = \varphi \circ \tr $. More precisely
\[
\phi(B^m)=\left\{
\begin{array}{l l}
0, & \text{if $m $ is odd}\\
1,  & \text{if $m $ is even.}
\end{array}
\right.
\]
Indeed, if $ m $ is odd, we have that
\[\displaystyle \phi(B^m) = \frac{1}{N}
\sum_{ \overrightarrow{i} \in [N]^m}
 \varphi\big(
b_{i_1, i_2} b_{ i_2, i_3} \cdots b_{ i_m, i_1}
 \big)
 \]
  and Theorem \ref{thm:3.1}(ii)
 give the conclusion since all the terms from the right-hand side of the equation above cancel.

 If $ m = 2p $ is even, then applying again Theorem \ref{thm:3.1}(ii) we get that
 \begin{align*}
 \phi(B^{2p})
 & =
 \sum_{ \overrightarrow{i}\in [ N ]^{2p}}
 \frac{1}{N}\prod_{k \in [p]}
 \varphi(b_{i_{2k-1}, i_{2k} }
 b_{i_{ 2k } i_{ 2k +1 } }  )\\
   = &
 \sum_{ \overrightarrow{i}\in [ N ]^{2p}}
  N^{ -( p +1)} \prod_{ k \in [p ] }
   \delta_{(i_{2k-1}, i_{2k} ), (i_{ 2k } i_{ 2k +1 } )}
   =
   \sum_{ \overrightarrow{i}\in [ N ]^{2p}}
     N^{ -( p +1)} \prod_{ k \in [p ] }
      \delta_{i_{2k-1},  i_{ 2k +1 } }\\
      &= N^{-(p+1)} \cdot \sharp
      \{ \overrightarrow{i} \in [ N ]^{2p} : \
       i_1 = i_3= \dots = i_{ 2p-1} \} =1.
 \end{align*}

 We shall show that the following analogue of Theorem \ref{thm:3.1} holds true for Bernoulli matrices.

\begin{thm} \label{thm:5.1}
 \emph{(i)} If $ B $ is a Bernoulli matrix from $ M_N ( \mathcal{A} ) $ and $ \sigma \in \mathcal{S} ( [ N^2]) $ satisfies equation (\ref{rel:1}) and $ t \circ \sigma = \sigma\circ t $
then $ B $ and $ B^\sigma $ are Boolean independent with respect to $ \varphi \circ \tr $.

\emph{(ii)} If $ ( B_N ) _{ N \in \mathbb{N}} $ is a sequence of Bernoulli matrices such that $ B_N \in M_N ( \mathcal{A}) $ and $ ( \sigma( N )  )_{ N \in \mathbb{N}} $ is a sequence of permutations such that, for each $ N $,  $ \sigma(N) $ is an element of $ \mathcal{B}([ N^2]) $ that satisfies
equation (\ref{rel:2}) and $ \sigma(N) \circ t = t \circ \sigma ( N ) $
 then $ S_N $ and $ { S_N}^{\sigma(N)} $ are asymptotically free.

\end{thm}

\begin{proof}

 We need to show that, for
 $ \omega_k \in \{ 1, \sigma \} $ such that $ \omega_k \neq \omega_{ k+1} $ and any  positive integers $ p_1, \dots, p_m $ we have that
 \[
 \phi\big( ( B^{\omega_1})^{p_1} \cdot
 (B^{\omega_2})^{p_2}\cdots (B^{\omega_m})^{p_m}\big) = \prod_{k=1}^m \phi\big((B^{\omega_k})^{p_k}\big)
 =\left\{
 \begin{array}{l l}
 1,& \text{if all $p_k$ are even}\\
 0, & \text{otherwise}.
 \end{array}
 \right.
 \]

 Let $ M = p_1 + p_2 + \dots + p_m $ and $ \overrightarrow{\sigma} = ( \sigma_1, \dots, \sigma_M) $ such that $ \sigma_j = \omega_k $ whenever $ p_1 + \dots + p_{k-1}  < j \leq p_1 + \dots + p_k $. Then
 \begin{align*}
  \phi\big( ( B^{\omega_1})^{p_1} \cdot
  (B^{\omega_2})^{p_2}\cdots (B^{\omega_m})^{p_m}\big)
  & =
  \phi( B^{\sigma_1}B^{\sigma_2} \cdots B^{\sigma_M}) \\
  &=
  \sum_{ \overrightarrow{i}\in[ N ]^M }
  \frac{1}{N}\varphi\big(
  [ B^{\sigma_1}]_{ i_1 i_2} [B^\sigma_2]_{ i_2 i_3} \cdots [ B^{\sigma_M}]_{ i_M i_1}
  \big).
 \end{align*}

If $ M $ is odd (case in which at least one $ p_k $ must be odd), then Theorem \ref{thm:3.1}(ii) gives that each term of the summation above equals zero, hence the conclusion.

If $ M $ is even, Theorem \ref{thm:3.1}(ii) gives that
\[
\varphi\big( B^{\sigma_1} \cdots B^{\sigma_M}\big)
=
\frac{1}{N} \sum_{ \overrightarrow{i}\in [ N ]^M}
[
\prod_{ 1\leq k \leq \frac{M}{2} }
\varphi \big(
[ B^{\sigma_{2k-1}}]_{i_{2k-1} i_{2k}} [B^{\sigma_{2k}}]_{i_{2k}i_{2k+1}}
\big)
]
\]
If all $ p_k $ are even, then $ \sigma_{2k-1} = \sigma_{2k} $ for all $ k =1, 2, \dots \frac{M}{2} $, henceforth
\begin{align*}
\phi\big( ( B^{\omega_1})^{p_1}\cdots (B^{\omega_m})^{p_m}\big)
  & =
  \frac{1}{N} \sum_{ \overrightarrow{i}\in [ N ]^M}
  [
  \prod_{ 1\leq k \leq \frac{M}{2} }
  \frac{1}{N}
  \delta_{\sigma_{2k-1}({i_{2k-1}, i_{2k}}),\sigma_{2k}(i_{2k+1},i_{2k})}
  ]
  \\
  = &
   N^{-(\frac{M}{2} + 1)}
  \sum_{ \overrightarrow{i}\in [ N ]^M}
  [
  \prod_{ 1\leq k \leq \frac{M}{2} }
  \delta_{i_{2k-1},i_{2k+1}}
  ]
  \\
  =&
  N^{-(\frac{M}{2} + 1)} \cdot
  \sharp\{ \overrightarrow{i}\in[N]^M : \ i_1 =i_3=\dots i_{M-1} \} =1.
  \end{align*}

  For the last part of the proof we will use an auxiliary result. Let $ R $ be an odd positive integer, $ \overrightarrow{\sigma} = ( \sigma_1, \dots, \sigma_R)$ with $ \sigma_k \in \{ 1, \sigma\} $ and $ \overrightarrow{i} = ( i_1, \dots, i_R) \in [N]^R $. For $ \alpha, \beta \in [ N ] $ , define $ \displaystyle c_{\overrightarrow{\sigma}}(\alpha, \beta)
     = \sum_{ \overrightarrow{i} \in [N]^R}
      c_{\overrightarrow{\sigma}}(\alpha, \beta, \overrightarrow{i}) $, where each
      $ c_{\overrightarrow{\sigma}}(\alpha, \beta, \overrightarrow{i}) $ equals
      \[
       \varphi \big( [ B^{\sigma_1}]_{ \alpha, i_1}
        [B^{\sigma_2}]_{i_1, i_2} \big)
        \cdot
        \varphi \big( [ B^{\sigma_3}]_{ i_2, i_3}
          [B^{\sigma_4}]_{i_3, i_4} \big)
          \cdots
          \varphi \big( [ B^{\sigma_{R-1}}]_{ i_{R-1}, i_R}
              [B^{\sigma_R}]_{i_R, \beta} \big).
      \]

      We shall show by induction on $ R $ that $c_{\overrightarrow{\sigma}}(\alpha, \beta)
      \leq 1 $
      for all
      $ \alpha$,  $\beta$, and $\overrightarrow{\sigma} $ as above.
      For $ R =1$, we get
      \begin{align*}
      c_{\overrightarrow{\sigma}}(\alpha, \beta)
      &=
      \sum_{j\in [N] }
      \varphi
      \big(
      [ B^{\sigma_1}]_{\alpha, j} [ B^{\sigma_2}]_{j, \beta}
      \big)
      =
      \frac{1}{N}
      \sum_{j\in [N] }
      \delta_{ \sigma_1(j, \alpha), \sigma_2 (j, \beta)}\\
      &=
      \frac{1}{N}
      \sharp\{ j \in [N]:\ \sigma_2^{-1}\circ \sigma_1 ( j, \alpha) = ( j, \beta) \}\leq 1.
      \end{align*}
  If $ R \geq 3 $, let  $\overrightarrow{\sigma}_{(1)} = (\sigma_3, \sigma_4, \dots, \sigma_R) $
  and
  $\overrightarrow{i}_{(1)} = (i_3,i_4, \dots, i_R) $. Then
  \[
  c_{\overrightarrow{\sigma}}(\alpha, \beta, \overrightarrow{i})
  =
  \varphi\big(
      [ B^{\sigma_1}]_{ \alpha, i_1} [ B^{\sigma_2}]_{i_1, i_2}
      \big)
       \cdot c_{\overrightarrow{\sigma}}( i_2, \beta, \overrightarrow{i}_{(1)} )
  \]
  therefore
  \begin{align*}
  c_{\overrightarrow{\sigma}}(\alpha, \beta)
        &
        =
    \sum_{ \overrightarrow{i} \in [ N]^R }
    \varphi\big(
    [ B^{\sigma_1}]_{ \alpha, i_1} [ B^{\sigma_2}]_{i_1, i_2}
    \big)
     \cdot c_{\overrightarrow{\sigma}}( i_2, \beta, \overrightarrow{i}_{(1)} )\\
     &
     =
    \sum_{ \overrightarrow{i}_{(1)} \in [N]^{R-1}}
    [
     c_{\overrightarrow{\sigma}_{(1)}}( i_2, \beta, \overrightarrow{i}_{(1)} )
     \cdot\sum_{i_1 \in [ N ] }
     \varphi\big(
         [ B^{\sigma_1}]_{ \alpha, i_1} [ B^{\sigma_2}]_{i_1, i_2}
         \big)
     ]\\
     &
     =
     \sum_{ \overrightarrow{i}_{(1)} \in [N]^{R-1}}
         [
          c_{\overrightarrow{\sigma}_{(1)}}( i_2, \beta, \overrightarrow{i}_{(1)} )
          \cdot
          c_{(\sigma_1, \sigma_2)} ( \alpha, i_2)
          ]
          \\
          &\leq c_{\overrightarrow{\sigma}_{(1)}}( i_2, \beta ).
      \end{align*}

  Finally, suppose that $ M $ is even but at least one $ p_k $ is odd. Through a circular permutation of (commutative) factors of the form
  $ \varphi \big(
  [ B^{\sigma_{2k-1}}]_{i_{2k-1} i_{2k}} [B^{\sigma_{2k}}]_{i_{2k}i_{2k+1}}
  \big) $
  we can suppose, without losing generality, that
   $p_1 = 1 $, that is $ \sigma_1 \neq \sigma_2$.
We the notations from above we have that
\begin{align*}
\varphi \big(
B^{\sigma_1} \cdots B^{\sigma_M}\big)
&=
\sum_{i_1, i_2, i_2 \in [ N ] } \frac{1}{N}
 \varphi\big( [ B^{\sigma_1}]_{ i_1, i_2}
 [B^{\sigma_2}]_{i_2, i_3}
 \big)
 \cdot
 c_{\overrightarrow{\sigma}_{(1)} } (i_3, i_1)
 \\
 \leq
 \frac{1}{N^2} &
 \sum_{i_1, i_2, i_2 \in [ N ] }
 \delta_{ \sigma_1(i_1, i_2), \sigma_2(i_3, i_2)}
 =
  \frac{1}{N^2}
  \sharp \{ (i_1, i_2, i_3) : \sigma(i_1, i_2) = (i_3, i_2)\}
\end{align*}
hence the conclusion.

\end{proof}



\section{Second order fluctuations for semicircular matrices and their transposes}

\subsection{Free cumulants of order higher than 3 }\label{section:highcu}

In this section, we always suppose $S_N=[c_{ij}]_{1\leq i,j\leq N}\in M_N(\mathcal{A})$ is a semicircular matrix, as defined in Section 3.
As before, we denote $ S_N^{(1)} = S_N $ and $ S_N^{(-1)} = S_N^t $, the matrix transpose of $ S_N $. To simply the notation, we will again omit the index $N$ if no confusion in the proof process.

Let $ l_1, l_2 \dots, l_r $ be positive integers and  $M = l_1 + \dots + l_r $. Let $ M(0) =0 $ and, inductively, $ M(s+1) = M(s) $.
Suppose $ \varepsilon $ is a mapping from $[ M ] $
to $ \{ (1), ( -1) \} $.
For each $ j =1, 2, \dots, r $ let
   \[ W_j = S_N^{\varepsilon( M_{ j-1}+1)}
   S_N^{\varepsilon( M_{ j-1}+2)} \cdots
   S_N^{\varepsilon( M_{ j})},
   \]
 and denote $ E = \Tr (W_1) \cdot \Tr(W_2) \cdots \Tr (W_r) $.

 For $ l_1, l_2, \dots, l_r $ as above, let $ \gamma $ be the interval partition on $ [ M ] $ with blocks $ B_1, B_2, \dots, B_r $ of lengths respectively $ l_1, \dots, l_r $ (i.e. $B_k = [ M_k] \setminus [M_{k-1}]$). We will identify a partition on $[ M ] $ with  a permutation of $ [ M ] $ by identifying each block with elements $ p_1< p_2< \dots < p_q $ with a cycle $(p_1, p_2, \dots, p_q)$. Furthermore, for $ \sigma $ a permutation on $ [ M ] $, we define $ \widetilde{\sigma} $ to be the permutation on $[ \pm M ] $ given by
  $ \widetilde{\sigma}(l) = - \sigma(l) $
  and $ \widetilde{\sigma}(-l) = \sigma(l) $ for each $ l \in [ M ] $. For example, if $ \gamma $ is as defined above, then
  \[ \widetilde{\gamma} = (-1, 2) ,( -2, 3), \dots, ( -M_1, 1), ( -M_1-1, M_1 +2), \dots, ( M, M_{ r-1}+1). \]

 For $ \oi= (i_1, i_{-1}, \dots, i_M, i_{ -M}) $ a $2M$-tuple with components from $ [ N] $ and $ \sigma $ a permutation on the set $[ \pm M ] = \{ 1, -1, \dots, M, - M \} $, we will say that $ \oi = \oi\circ \sigma $ if $ i_s = i_r $ whenever $ r $ and $ s $ are in the same cycle of $ \sigma $.

 Finally, denote
 $c^{\varepsilon(j)}_{ k ,l } $  the
  $ (k, l) $ -entry of
   $ S_N^{\varepsilon(j) }$, i.e.  it equals $ c_{ k, l } $ if $ \varepsilon(j) = (1) $, respectively
     $ c_{  l, k } $ if $ \varepsilon(j) = (-1) $.

    With these notations, we have that
   \[ \varphi(E) = \sum_{\substack{\oi \in [N]^M \\ \oi =\oi \circ \widetilde{\gamma}}} \varphi( c^{\varepsilon(1)}_{ i_1 i_{ -1} }
   \cdots c^{\varepsilon(M)}_{ i_M i_{-M}}  )
   \]

 Further, Theorem \ref{thm:wick}(i) gives that
 \[
 \varphi(E) = \sum_{ \pi \in NC_2 (M)} v (E, \pi)
 = \sum_{ \pi \in NC_2 (M)}
 \sum_{\substack{\oi \in [N]^{2M} \\ \oi =\oi \circ  \widetilde{\gamma}}} w ( E, \pi,\oi)
 \]
 for
 $ \displaystyle v(E, \pi)
 = \sum_{
 \substack{\oi \in [N]^{2M} \\ \oi =\oi \circ  \widetilde{\gamma}}
 }
 w ( E, \pi,\oi) $ and
 $  \displaystyle w ( E, \pi,\oi) =
  \prod_{ ( k, l) \in \pi }
  \varphi( c^{ \varepsilon(k)}_{ i_ki_{ -k}}
   c^{ \varepsilon(l)}_{ i_l i_{ -l }}) $.

    In particular, with the notations from Section \ref{section:22},
    \[
    w ( E, \pi,\oi) =
     \kappa_{\pi}
     [
     c_{i_1, i_{ -1}}, c_{ i_2, i_{-2}}, \dots, c_{i_M, i_{ -M}}
     ].
    \]

  We will identify  $ \pi \in NC_2(M) $ to a permutation  $ \widetilde{\pi} $ on $[ \pm M ] $ via $ \widetilde{\pi}(k) = - l $ if $ (k, l) $ is a block of $ \pi $.
  Also, the   induces a permutation $ \epsilon $ on $ [ \pm M ] $

 \begin{remark}\label{rem:cycle}
  Let $ \epsilon $ be the permutation on $ [ \pm M ] $ induced by the mapping $ \varepsilon $  via
      \[
      \epsilon(k) = \left\{
      \begin{array}{ll}
      k, & \text{if}\ \varepsilon(|k|) = (1)\\
      -k, & \text{if}\ \varepsilon(|k|) = (-1).\\
      \end{array}
       \right.
      \]
 Then, with the notations above,
 \[
w(E, \pi, \oi)
          = N^{  - \frac{M}{2} }
                       \delta_{ \oi, \oi \circ \epsilon \widetilde{\pi}\epsilon}
 \]
 and
 \[
 v(E, \pi)
 =
           \sum_{  \oi \in [N]^{2M} }
           N^{  - \frac{M}{2} }
   \delta_{ \oi, \oi \circ ( \widetilde{\gamma} \vee \epsilon \widetilde{\pi}\epsilon ) }
   =N^{ \sharp ( \widetilde{\gamma} \vee \epsilon \widetilde{\pi} ) - \frac{M}{2} }
   \]
  \end{remark}
 \begin{proof}
 The definition of the mapping $ \epsilon $ gives that
 $ c^{ \varepsilon(k)}_{ i_ki_{ -k}}
      =
      c_{ i_{ \epsilon(k)}, i_{\epsilon(-k)}}
      $,
      hence
    \begin{align*}
      w(E, \pi, \oi)
           = &
           \prod_{ ( k, l) \in \pi }
             \varphi( c_{ i_{\epsilon(k)}i_{ \epsilon(-k)}}
              c_{ i_{\epsilon(l)} i_{\epsilon( - l ) }})
              =\prod_{ ( k, l) \in \pi } \frac{1}{N}
              \delta_{ i_{ \epsilon(k)}, i_{\epsilon(-l)}}
              \delta_{i_{ \epsilon( -k)}, i_{\epsilon(l)} }\\
              = & N^{  - \frac{M}{2} } \prod_{ k \in [ \pm M ] } \delta_{ i_{ \epsilon(k)}, i_{\epsilon\pi (k)} }
              =
              N^{  - \frac{M}{2} }
              \delta_{ \oi, \oi \circ \epsilon \widetilde{\pi}\epsilon}.
      \end{align*}
 Henceforth,  the equation
 $  \displaystyle v(E, \pi)
  = \sum_{
  \substack{\oi \in [N]^{2M} \\ \oi =\oi \circ \widetilde{\gamma}}
  }
  w ( E, \pi,\oi) $  give
   \[
      v(E, \pi) =  \sum_{
       \substack{\oi \in [N]^{2M} \\ \oi =\oi \circ \widetilde{\gamma}}
       }
       N^{  - \frac{M}{2} }
                    \delta_{ \oi, \oi \circ \epsilon \widetilde{\pi} \epsilon}
                    =
         \sum_{  \oi \in [N]^{2M} } N^{  - \frac{M}{2} }
         \delta_{ \oi , \oi \circ \widetilde{\gamma} }
         \cdot
         \delta_{ \oi, \oi \circ\epsilon\widetilde{\pi}\epsilon}
 \]
    But $ \oi $ is constant on both the cycles of $ \widetilde{\gamma} $ and  of $ \epsilon \widetilde{\pi} \epsilon $ if it is constant on the blocks of
     $ \widetilde{\gamma} \vee \epsilon \widetilde{\pi}\epsilon $, so
    \[
    v(E, \pi) = N^{  - \frac{M}{2} } \cdot
    \sharp\{ \oi \in [N]^{2M}:\ \oi = \oi \circ( \widetilde{\gamma} \vee \epsilon \widetilde{\pi} \epsilon)\}
    \]
    hence the conclusion.
 \end{proof}

  \begin{remark}\label{rem:v}
   With the notations above, suppose that $ \pi \in NC_2 (M ) $ contains the block $ (k, k+1) $ such that  $ k $ and $ k+1 $ are in the same block of $ \gamma $ (i.e.   $ S^{ \varepsilon(k)}_N  $ and $ S^{ \varepsilon(k+ 1 )}_N $ are factors in the same $ W_j $). Let $ \pi_{(k)} $ be the noncrossing partition on the ordered set $ \{ 1, 2, \dots, k-1, k+2, \dots, M \} $ obtained by deleting block $ ( k, k+1) $ from $ \pi $ and let $  E_{(k)} $ be the product obtained by deleting the factors $ S^{ \varepsilon(k)}_N  $ and $ S^{ \varepsilon(k+ 1 )}_N $ from the development of $ E $. Then:
   \[ v(E, \pi) = \left\{
  \begin{array}{ll}
  v ( E_{(k)} \pi_{(k)}) &  \text{if} \ \varepsilon(k) = \varepsilon(k+1)\\
  \frac{1}{N}v ( E_{(k)} \pi_{(k)}) &  \text{if}\  \varepsilon(k) \neq \varepsilon(k+1).
  \end{array}
    \right.
    \]
  \end{remark}

  \begin{proof}
 Since $ ( k , k +1) \in \pi $, we have that $  w ( E, \pi,\oi) = 0 $ unless
 \begin{equation}\label{eq:varphi}
  \varphi ( c^{\varepsilon(k)}_{ i_k i_{-k}}
 c^{\varepsilon(k+1)}_{ i_{k+1} i_{-(k+1)}})
 \neq 0 .
 \end{equation}

 If $ \varepsilon(k) = \varepsilon(k+1) $,  the  equation (\ref{eq:varphi}) implies that $ i_k = i_{ -(k+1)} $ and $ i_{-k} = i_{ k+1} $.

 Furthermore,
 since
  $ k $ and $ k + 1 $ are in the same block of $ \gamma $, we get that
  $ \gamma(k) = k+1 $. Hence
   $ \oi = \oi \circ \widetilde{\gamma} $
  gives that
   $ i_{ \widetilde{\gamma}^{-1} (k)} = i_k = i_{ -(k+1)} =
    i_{  \widetilde{\gamma}( - (k +1))} $.

 Therefore, for
 $ \oi_{(k)} = ( i_1, i_{-1}\dots i_{k-1}, i_{ -(k-1)}, i_{ k+2}, i_{ -(k+2)}, \dots, i_{-M}) $ and for
 $ \gamma_{(k)} $
 obtained from $ \gamma $ by deleting the elements $ k, k+1 $  we get that
  $  w ( E, \pi,\oi) = 0 $
   unless
   $ \oi_{(k)} = \oi_{(k)} \circ \widetilde{\gamma_{(k)}} $.
 Hence
 \[
 v(E, \pi) =
  \sum_{ \oi = \oi \circ  \widetilde{\gamma} } w (E, \pi, \oi) = \sum_{  \oi_{(k)} = \oi_{(k)} \circ \widetilde{\gamma_{(k)}}}
  \sum _{ \substack{i_k = i_{ -(k+1)}\\ i_k = i_{ -(k+1)} = i_{  \widetilde{\gamma}( -k -1)}  }}
  w(E, \pi, \oi)
 \]

 But  $ w(E, \pi, \oi) = w(E_{(k)}, \pi_{(k)}, \oi_{(k)}) \cdot \varphi ( c^{\varepsilon(k)}_{ i_k i_{-k}}
  c^{\varepsilon(k+1)}_{ i_{k+1} i_{-(k+1)}}) $
  and the last factor equals $ \frac{1}{N} \delta_{ i_k i_{ -(k+1)} }
   \delta_{i_{ -k} i_{ k+1}}  $. Also,  $ i_{  \widetilde{\gamma}( -k -1)} $ is a component in $ \oi_{(k)} $, therefore
   \begin{align*}
   v(E, \pi) &= \sum_{  \oi_{(k)} = \oi_{(k)} \circ \widetilde{\gamma_{(k)}})}  \sum_{i_k = i_{ -(k+1)}} \frac{1}{N} w(E_{(k)}, \pi_{(k)}, \oi_{(k)}) \\
    &= \sum_{  \oi_{(k)} = \oi_{(k)}\circ \gamma_{(k)}}w(E_{(k)}, \pi_{(k)}, \oi_{(k)})
        = v ( E_{(k)} \pi_{(k)}).
   \end{align*}
  \end{proof}

  If $ \varepsilon(k)  \neq \varepsilon(k+1)$, then equation (\ref{eq:varphi}) implies that $ i_{k} = i_{ k+1} $ and $ i_{ -k} = i_{ -(k+1)} $.
  Then  $ \gamma(k) = k +1 $ and $ \oi = \oi \circ  \widetilde{\gamma} $
     give that $ i_{ \widetilde{\gamma}^{-1}(k)}$, $i_{k} $, $ i_{-k} $, $ i_{ k+1} $ , $ i_{ -(k+1)} $, $ i_{ \widetilde{\gamma}( - k -1)} $ are all equal. Hence the condition $ \oi_{(k)} = \oi_{(k)} \circ \widetilde{ \gamma_{(k)} } $ is equivalent to equation (\ref{eq:varphi}) and $ \oi = \oi \circ  \widetilde{\gamma} $, so
  \[
  v(E, \pi) = \sum_{  \oi_{(k)} = \oi_{(k)} \circ \widetilde{\gamma_{(k)}} }   \frac{1}{N} w(E_{(k)}, \pi_{(k)}, \oi_{(k)})
  \]
  hence the conclusion.
 \begin{cor}\label{cor:tr}
 Suppose that $ \varepsilon(v+1) =\varepsilon(v+2)= \dots = \varepsilon(v+p)  $ for some $v $ and $ p $ such that $ M_{j-1}\leq v < M_j - p $. Let $ \mathcal{D}(v, p) $ be the set of all $ \pi  \in NC_2(M) $ such that the set $ \{ v+1, v+2, \dots, v+p \} $ is a union of blocks  action of $ \pi $ (in particular, $  p$ must be even). Then
 \[
 \sum_{ \pi \in \mathcal{D}(v, p)} v(E, \pi) = \sum_{\sigma\in NC_2(M-p)} v ( E_{(v, p)}, \sigma) \cdot \phi  \big( (S_N^{\ \varepsilon(v+1)})^p \big).
 \]
 \end{cor}
   \begin{proof}
   Note that each $ \pi \in \mathcal{D}(v, p) $ is uniquely determined by its restriction to $ \{ v+1, \dots, v+p \} $ (which is an element from $ NC_2 (p) $ ) and by its restriction to $ [ M ] \setminus \{ v+1 , \dots, v+p \} $ (which is an element of $ NC_2(M-p)$). Moreover, given $ \sigma_1 \in NC_2(p) $ and $ \sigma_2 \in NC_2(M-p) $, there is a unique $ \pi \in \mathcal{D}(v, p) $ such that its restrictions to $\{ v +1, \dots, v+p\} $, respectively $ [M]\setminus\{ v +1, \dots, v+p\}  $  equal $ \sigma_1$, respectively $ \sigma_2 $.

 Thus, applying Remark \ref{rem:v}, we obtain
 \[
 \sum_{ \pi \in \mathcal{D}(v, p)} v(E, \pi) =
 \sum_{\sigma_1 \in NC_2(M-p)} \sum_{ \sigma_2 \in NC_2(p) } v (E_{(v, p)}, \sigma_1).
 \]
 But $ S_N $ and $ S_N^t $ are both semicircular matrices, so $ \phi  \big( (S_N^{\ \varepsilon(v+1)})^p \big) = \sharp NC_2(p) $, hence the conclusion.
   \end{proof}

 Let $ \beta $, respectively $ \pi $ be an interval partition, respectively non-crossing pair partition of $ [ M ] $. Two blocks $ B_1 $ and $ B _2 $ of $ \beta $ are said to be connected by $ \pi $ if there exist some $ i \in B_1 $ and $ j \in B_2 $ such that $ (i, j) \in \pi $. A block of $ \beta $ is said to be left invariant by $ \pi $ if it is not connected by $ \pi $ to any other block of $ \beta $. We will denote by $ NC_2(M, \beta, s) $ the set of all elements of $ NC_2 (M ) $ that leave invariant all blocks of $ \beta $. Also, we will denote by $ NC_2(M, \beta, c) $ the set of all elements $ \pi $ of $ NC_2(M) $ such that for any two blocks $ B_1, B_2 $ of $ \beta $ there exist a sequence $ D_1, D_2, \dots, D_p $ of blocks of $ \beta $ such that $ D_1 = B_1 $, that $ D_p = B_2 $ and that $\pi $ connects $ D_i $ and $ D_{ i + 1} $ for each $ i $.

  \begin{remark}\label{rem:61}
  With $ l_j, M(j), M $ as above $ ( 1 \leq j \leq r ) $,
  let $c_1, \dots, c_M $ be elements of $ \mathcal{A} $ such that $\{ \Re{c_i}, \Im c_i:\  1 \leq i \leq M \} $ is a family of free centered semicircular variables.  Take $ C_j = c_{M(j-1)+ 1} c_{M(j-1) + 2} \cdots c_{M(j)} $ and let $ \beta $ be the interval partition of $ [M]$ with blocks of lengths $ l_1, l_2, \dots, l_r $ in this order. With the notations above, we have that
  \[
  \kappa_p( C_1, C_2, \dots, C_r) =
   \sum_{ \pi \in NC_2( M, \beta, c) }
   \prod_{ (i, j) \in \pi } \phi(c_i c_j).
  \]
  \end{remark}
 \begin{proof}
  For $ r =1 $, the relation is a particular case of Theorem \ref{thm:wick}(i) (i.e. the Free Wick formula). Suppose now that the result holds true for $r < R $. Then Theorem \ref{thm:wick}(i) gives that
  \begin{align*}
  \varphi(C_1\cdots C_r)&=  \varphi(c_1 c_2 \cdots c_M)\\
  & = \sum_{ \pi \in NC_2 (M, \beta, c) }
  \prod_{ (i, j) \in \pi } \phi(c_i c_j) +
  \sum_{\substack{\pi \in NC_2 (M)\\
   \pi \notin NC_2 (M, \beta, c)}}
  \prod_{ (i, j) \in \pi } \phi(c_i c_j),
  \end{align*}
  while the moment-free cumulant recurrence (\ref{eq:freecum}) gives that
  \[
  \varphi(C_1\cdots C_p) = \kappa_p (C_1, \dots, C_p) +
  \sum_{\substack{\pi \in NC(p)\\ \pi \neq 1_p}}
  \kappa_\pi[ C_1, \dots, C_p]
  \]
  and the induction hypothesis gives that the last terms from the right-hand sides of the two equations above are equal, hence the conclusion.

 \end{proof}

   \begin{thm}\label{thm:highc}
   With the notations above,
   \[
   \lim_{ N \longrightarrow \infty}
   \kappa_r \big(
   \Tr(W_1), \Tr(W_2), \dots, \Tr(W_r)
    \big) = 0
 \]
   \end{thm}

   \begin{proof}
    As before, let  $ \beta $ be the interval partition on $ [M] $ with blocks $ B_1, B_2, \dots, B_r $ where $ B_k = ( M_{k-1}+1, M_{k-1}+2, \dots, M_{k} )$ for each $ k \in [r]$. We will denote by $ \pm B_k $ the set $ \{ q \in \pm [ M ] : \  | q | \in B_k \} $. Also, as before, let
    $ E = \Tr(W_1)\cdot \Tr(W_2)\cdots \Tr(W_r)$ and let $ \gamma$ be the permutation on the set
      $ [ \pm M ] $ with cycles $ \{ ( -k,k+1) : k \neq M_j , 1 \leq j \leq r\} \cup \{ (-M_j, M_{ j-1}+1): 1 \leq j \leq r \}  .$

     With this notation, Remark \ref{rem:61} gives that
    \[
    \kappa_r\big( \Tr(W_1), \Tr(W_2), \dots, \Tr(W_r)\big) =
    \sum_{ \pi \in NC_2(M, \beta, c) }
    v(E, \pi)
    \]
    hence it suffices to show that
     $ \displaystyle \lim_{ N \longrightarrow \infty}  v(E, \pi) = 0 $
     for each $ \pi \in NC_2(M, \beta, c) $. Furthermore, utilising Remark \ref{rem:v}, it suffices to show the result for partitions $ \pi $ that also have the property that do not pair consecutive elements from the same block of $ \beta $. Since $ \pi $ is non-crossing, the last property is equivalent to $ \pi $ not pairing elements from the same block of $ \beta $.

     Consider then $ \pi \in NC_2(M) $ that connects all blocks of $ \gamma $, but do not pair elements from the same block of ${\gamma}$.
     Since, from Remark \ref{rem:cycle}
      \[
      v(E, \pi)
      =
                N^{ \sharp ( \gamma \vee \epsilon \widetilde{\pi}\epsilon ) - \frac{M}{2} }
     \]
      it suffices to show that
      $ \sharp ( \gamma \vee \epsilon \widetilde{\pi}\epsilon ) < \frac{M}{2} $ for all $\pi $ as above.

     For $ p \in [ \pm M ] $ and  $ \sigma $ a partition of $ [ \pm M ]$, let us denote by
     $ p _{ / \sigma} $ the block of $ \sigma $ that contains $ p $.

      Note that the relation $ \oi = \oi \circ \gamma $ gives that if $ B_k $ is a block of $ \gamma $ and $| p |\in B_k $, then
       $ p_{ / \gamma } $
        has at least two elements in $ \pm B_k $. Indeed, if $ p > 0 $, then $ \gamma(p) = -{p-1} $  if  $ p \neq M_{k-1}+ 1 $, respectively
      $ \gamma(p) = -M_k $ if $ p = M_{ k-1}+1 $ (see also the diagram below). The situation is similar for $ p < 0 $.

 \begin{center}
  \includegraphics[height=17mm]{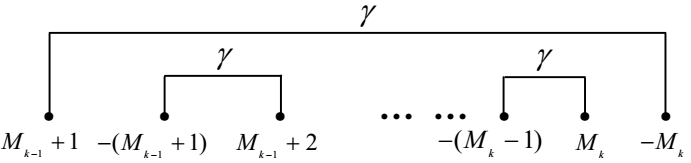}
  \end{center}

       If $ (v, u)\in \pi $ with $ v \in B_k $ and $ u \in B_l $ such that $ l \neq k $, then
        $ \oi = \oi \circ \epsilon \widetilde{\pi} \epsilon $
        gives that
      $ w ( E, \pi,\oi) \neq 0 $,
      in particular
      \[
       0 \neq \varphi (
      c^{ \varepsilon(v)}_{i_v i_{ -v}} c^{\varepsilon(u)}_{i_u i_{ -u}})
      = \frac{1}{N}\delta_{ i_{ \epsilon(k)}, i_{\epsilon(l)}}
                    \delta_{i_{ \epsilon( -k)}, i_{\epsilon(-l)} }
      .
     \]
       Hence $ i_v = i_{-u} $ if $ \varepsilon(u) \neq \varepsilon(v) $ or $ i_v = i_u $ if $ \varepsilon(u) = \varepsilon(v) $, that is $ v_{ / \epsilon\pi} $
        has elements in both $ \pm B_k $ and $ \pm B_l $. Analogously, same is true for $ -v_{ / \epsilon\pi} $. (See also the diagram below)

\begin{center}
\includegraphics[height=23mm]{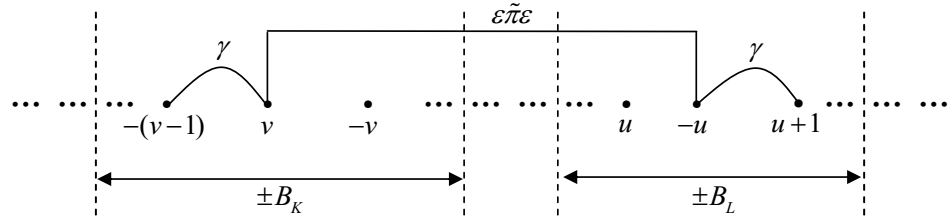}
\end{center}

       Since $ v_{/ \gamma} $ and $ v_{ / \epsilon\widetilde{\pi}\epsilon} $ are subsets of
        $ v_{ \gamma \vee \epsilon\widetilde{\pi}\epsilon} $, it follows that the each cycle of $ \gamma \vee \epsilon\widetilde{\pi}\epsilon $ has at least 4 elements.
      Since  the set $ [ \pm M ] $ has $ 2 M $ elements, it suffices to show that there exist a cycle with strictly more than 4 elements.

        $ \pi $ connects all blocks of $ \beta $, so at least one block is connected to at least two other distinct ones. Since $ \pi $ is noncrossing and does not connect elements from the same block, we can suppose that there exist some $ v \in B_k $ and some $ t \in B_l $  and $ s \in B_p $ such that $ v + 1 $ is also an element of $ B_k $ and $ (v, t), (v+1, s) $ are blocks in $\pi $. (see again the diagram below).

  \begin{center}
  \includegraphics[height=24mm]{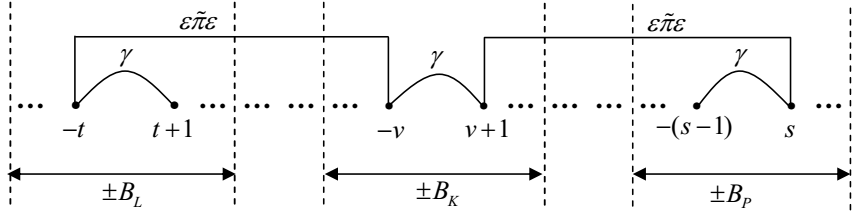}
  \end{center}

        As seen before, $ (v, t) \in \pi $ implies that $ v_{/ \epsilon\widetilde{\pi}} \epsilon $  and
        $ -v_{/ \epsilon\widetilde{\pi}}\epsilon $ have elements in both $ B_k $ and $ B_l $. Similarly,
         $ v+1_{/ \epsilon\widetilde{\pi}\epsilon} $
         has elements in both $B_k $ and $ B_p $. But $ \oi = \oi \circ \gamma $ gives that $ v+1_{ /\gamma} = v_{ /\gamma} $
          or
          $ v+1_{/ \gamma} = -v_{/ \gamma} $
          that is
          $ v+1_{/ \gamma \vee\epsilon\widetilde{\pi}\epsilon} $
           has at least 6 elements, (at least two in each
         $\pm B_k$, $\pm B_l $, and $ \pm B_p $), hence the conclusion.

   \end{proof}

\subsection{Second order free cumulants}${}$\\

\begin{lemma}\label{lemma:tr}
 Let $  s $, $ r $ be two positive integers and
  $ \omega : [ s  + r ] \longrightarrow \{ (1), ( -1) \} $ be a mapping such that
  is such that  $ \omega(p) \neq \omega (p + 1 ) $
  for $ p \in [ s+r-1] \setminus \{s \}$.

 For each $ 1 \leq i \leq s + r $,  let
   $ T_{i, N} $ be a polynomial in
    $ S_N^{ \omega_i} $
     with complex coefficients and denote
 $ W_{i, N}^\circ  = T_{i, N} - \phi ( T_{i, N} ) $.
 Then
 \[
 \phi\big(
 W_{1, N}\cdots W_{s, N} \cdot
 W_{s+1, N} \cdots
 W_{s+r, N} \big)
 =
 \delta_{ s, r}
 \prod^s_{j=1}
 \lim_{N\rightarrow\infty}
 \phi(W_{s + 1 -j, N}
 \cdot W_{s+j, N}).
 \]
\end{lemma}
\begin{proof}
 From Corollary \ref{cor:3.3},
 $ W_{ i, N} $ and  $ W_{ i+1, N} $ are asymptotically   free, unless $ i = s $.

 If $ \omega_s \neq \omega_{ s + 1 } $, then the definition of free independence gives that the left-hand side of the equation cancels, as well as the factor corresponding $ j = s $ in the right-hand side.

 If  $ \omega_s = \omega_{ s + 1 } $, then
 \begin{equation}\label{eq:5.1}
 \lim_{ N \longrightarrow \infty}
 \phi( W_{ 1} \cdots W_{ s+r})
 = \lim_{ N \longrightarrow \infty}
 \phi( W_{ s}\cdot W_{ s+1})\cdot
 \phi( W_{ 1} \cdots
 W_{ s-1} \cdot W_{ s+1}
 \cdots W_{ s+r})
 \end{equation}
because, for $ \beta = W_s W_{s+1} - \phi (W_s^\circ \alpha_{s+1} ) $, we get that
\begin{align*}
\phi( W_1\cdots W_{s+r})
=
 \phi( W_{ s} \cdot W_{ s+1} )\cdot
 &
  \phi( W_{ 1} \cdots
  W_{ s-1} \cdot W_{ s+1}
  \cdots W_{ s+r}) \\
  & + \phi( W_1 \cdots W_{ s-1} \cdot \beta \cdot W_{ s+2} \cdots W_{ s+r})
\end{align*}
and the second term of the right-hand side from above cancels asymptotically from the definition of free independence and Corollary \ref{cor:3.3}.

If $ s = r $, the conclusion follows from (\ref{eq:5.1}) and an inductive argument.

If $ s = r + w $ with $ w > 0 $, then ( \ref{eq:5.1}) gives
\begin{align*}
\lim_{N \longrightarrow\infty}
\phi(W_1 \cdots W_{ s + r })
=
\lim_{N \longrightarrow\infty}
\phi( W_1 \cdots W_w) \cdot \prod_{ j =1}^r
\phi( W_{ s+1-j}W_{s+j})
\end{align*}
and, applying again Corollary \ref{cor:3.3}, the first factor of the right-hand side cancels asymptotically.  The case $ r > s $ is analogous.
\end{proof}

 The main result of this section is the following:

 \begin{thm}\label{thm:2c}
  Let $ s$,  $ r $ be two positive integers and
  $ \omega : [ s+r ] \longrightarrow
   \{ (-1), ( 1 )\} $
   be a mapping such that
   $ \omega(p) \neq \omega (p+1) $
   for $ p \in [ s+r - 1 ] \setminus \{ s \} $.

   For each $ k \in [ s + r ] $, suppose that $ P_k $ is a polynomial with complex coefficients and denote
   $ \alpha_k^{(N)} = P_k ( S_N^{\omega(k)}) -
   \phi \big(P_k ( S_N^{\omega(k)}) \big).$
   \begin{enumerate}
   \item[(i)]If $ \omega (s) = \omega( s + 1) $, then
   \[
      \lim_{ N \longrightarrow \infty}
      \kappa_2 \big(
      \Tr (\alpha_1^{(N)} \cdots \alpha_s^{(N)}),
      \Tr (\alpha_{s+1}^{(N)} \cdots \alpha_{s+r}^{(N)})
       \big) =
       \delta_{s,r}
       \prod_{j=1}^s
       \lim_{N\rightarrow\infty}
       \phi(\alpha^{(N)}_j\alpha^{(N)}_{s+r+1-j}).
      \]
   \item[(ii)] If
   $ \omega (s) \neq \omega( s + 1) $
   and $ s \neq r $ or $ s, r \geq 3 $, then
   \[
   \lim_{ N \longrightarrow \infty}
         \kappa_2 \big(
         \Tr (\alpha_1^{(N)} \cdots \alpha_s^{(N)}),
         \Tr (\alpha_{s+1}^{(N)} \cdots \alpha_{s+r}^{(N)})
          \big) = 0.
   \]
     \end{enumerate}
  \end{thm}

  \begin{proof}

  Since the free cumulants are multilinear, it suffices to prove the result for $ P_k $ monomials, that is
  \[
  \alpha_k = ( S^{\omega(k)})^{l_k} - \phi\big( ( S^{\omega(k)})^{l_k} \big)
  \]
  for $l_1, l_2, \dots l_{ s+r} $ some positive integers.

  As before, let $M_0 = 0 $ and
   $ M_j = M_{j-1} + l_j $
    for $ j \in [ s+r] $.
     Let $ m = l_1 + \dots l_s = M_s $,
     let $ n = l_{ s+1} + \dots + l_{ s + r} = M_{s+r} - M_s$ and $ M = m +n = M_{s+r} $.

  Define $ \varepsilon: [ M ] \longrightarrow \{ (1), (-1)\} $ via $ \varepsilon(p) = \omega(k) $ whenever $ M_{ k-1} < p \leq M_k $.

  Also, denote
  \begin{align*}
  W_1 &
  =
  ( S^{ \omega(1)})^{ l_1}\cdot
  ( S^{ \omega(2)})^{ l_2}  \cdots (S^{\omega(s)})^{l_s} = S^{\varepsilon(1)} S^{\varepsilon(2)} \cdots S^{\varepsilon(m)}\\
  W_2 &
   =
   ( S^{ \omega(l_s+1)})^{ l_{s+1}} \cdots (S^{\omega(s+r)})^{l_{s+r}} = S^{\varepsilon(m+1)} S^{\varepsilon(m+2)} \cdots S^{\varepsilon(m+n)}
  \end{align*}
an $ E = \Tr(W_1) \Tr(W_2) $, respectively $ F=\Tr(W_1W_2) $.

Let $ \beta $ be the interval partition on
 $ [ M ] $ with blocks
 $ B_1, \dots, B_{s+r} $
  of lengths
  $ l_1, l_2, \dots, l_{ s+r} $
  in this order.
   The definitions of $ W_1 $ and $ W_2 $ give that if $ k, l \in B(i) $, then
$ \varepsilon(k)= \varepsilon(l) = \omega(i) $.

Corollary \ref{cor:tr} gives that
\begin{equation}\label{eq:F1}
 \varphi\circ \Tr( \alpha_1 \cdots \alpha_{ s+r} ) =
\sum_{
\pi \in NC_2(M) \setminus NC_2( M, \gamma, s)
 }
  v ( F, \pi).
\end{equation}

Let $ \gamma $ be the interval partition on $ [ M ] $ with two blocks $ D_1, D_2 $ of lengths $ m $ and $n $ in this order.  In particular, if $ k \in D_i $, then $ S^{\varepsilon(k)} $ is a factor in the development of $ W_i $. We will prove by induction on $ M $ the following property
\begin{equation}\label{eq:F2}
\left\{
\begin{array}{l}
  v (F, \pi) = O(N) \ \text{for all}\ \pi \in NC_2(M) \setminus NC_2(M, \gamma, s)\\
 v( F, \pi) = O(1) \ \text{if}\ \pi \notin   NC_2( M, \gamma, c) \setminus NC_2(M,\beta, s).
\end{array}
\right.
\end{equation}

If $  M $ is odd, the property is vacuously true. If $ M = 2 $, the first equation follows from $ \varphi \circ \Tr(S^2) = N $ while the second is again vacuously true. For the induction step, first note that, from Remark \ref{rem:v},  we can suppose that $\pi $ does not pair consecutive elements from the same block of $ \beta $, henceforth, since $ \pi $ is non-crossing, it does not pair any elements from the same block of $ \beta $. Then, applying Lemma \ref{lem:3.2}, $ \pi $ does pair two consecutive elements  $ k, k+ 1 $ of $ [ M ]$ that are in different blocks, $ B_i $ and $ B_{ i + 1} $ of $ \gamma $.
If $ \pi \notin NC_2(M, \gamma, c) $, then  $ B_i $ and $ B_{i +1} $ are in the same block of $ \gamma$, so $ \omega(i) \neq \omega( i + 1) $. Then the induction hypothesis and the second part of Remark \ref{rem:v} imply that $ v( F, \pi) = \frac{1}{N} O(N) =  O(1) $. If $ \pi \in NC_2(M, \gamma, c) $, from the argument above it suffices to show the property for $\pi $ pairing only elements from different blocks of $ \beta $, that is for $ m = n $ and $ \pi = \{ ( j, M - j + 1): j \in [ m ]\} $.
Then we apply Remark \ref{rem:v} to the block $(m, m+1) $ and note that, with the notations from Remark \ref{rem:v}, $ \pi _{(m+1)}  \in NC_2(M, \gamma, c) \setminus NC_2(M, \beta, s) $, hence the conclusion follows from the induction hypothesis.

Equation (\ref{eq:F1}) and property (\ref{eq:F2}) give that
\begin{equation}\label{eq:F}
\lim_{ N \longrightarrow\infty} \frac{1}{N} \varphi\circ \Tr
( \alpha_1 \cdots \alpha_{ s+r})
= \lim_{ N \longrightarrow \infty}
\sum_{
\pi \in NC_2(M, \gamma, c) \setminus NC_2( M, \beta, s)
 }
  \frac{1}{N}v ( F, \pi).
\end{equation}
 On the other hand, Remark \ref{rem:61} and Corollary \ref{cor:tr} give that
 \begin{equation}\label{k2}
 \kappa_2 (\Tr( \alpha_1 \cdots \alpha_s),
 \Tr( \alpha_{s+1} \cdots \alpha_{s+r}))
 =
\sum_{
 \pi \in NC_2(M, \gamma, c) \setminus NC_2( M, \beta, s)
  } v( E, \pi).
 \end{equation}

Let us first consider the case $ \omega(s) = \omega(s+1) $, that is $ \varepsilon(m) = \varepsilon(m + 1) $. According to Lemma \ref{lemma:tr}, it suffices to show that
\[
\lim_{ N \longrightarrow\infty}
 \kappa_2\big(
  \Tr( \alpha_1\cdots \alpha_s),
  \Tr(\alpha_{s+1}  \cdots \alpha_{ s+r})
  \big)
= \lim_{ N \longrightarrow \infty}
\frac{1}{N}
\varphi \circ \Tr ( \alpha_1 \cdots \alpha_{ s+r}),
\]

 hence it suffices to show that
  \begin{equation}\label{eq:vv}
   v(E, \pi) = \frac{1}{N} v( F, \pi)
  \end{equation}
  for all $ \pi \in Nc_2(M, \beta, c) \setminus NC_2( M, \gamma, s) $.

  We shall prove (\ref{eq:vv}) by induction on $ M $. If $ M =2 $, then $s= r = 1 $ and $ \pi = (1, 2) $. Also $ \varepsilon(1) = \varepsilon(2) $, hence
  \[
  v( F, \pi) = \sum_{ i_1, i_2 \in [ N ] } \varphi ( c_{ i_1 i_2}^{\varepsilon(1)} c_{ i_2 i_1}^{ \varepsilon(1)}) = N^2 \cdot \frac{1}{N} = N
  \]
  while
  \[
   v( E, \pi) = \sum_{ i_1 \in [ N ] } \varphi( c_{ i_1 i_1}^{ \varepsilon(1)} c_{ i_1 i_1}^{\varepsilon(1)} ) = N \cdot \frac{1}{N} = 1
   \]
  so the property is trivial.

    For the induction step, suppose first that if $ \pi $ pairs two elements, $ k $ and $ l $ from the same block of $ \beta $. Then the restriction of $ \pi $ to the set $ \{ k, k + 1, \dots, l \} $ is a noncrossing pair partition, hence contains a block consisting on consecutive elements. Then (\ref{eq:vv}) follows from the application of Remark \ref{rem:v} in both sides,  and from the induction hypothesis.

      We can suppose though that all blocks of $\pi $ contain one element from each block of $ \gamma $, that is one element from $ [ m ] $ and one from $ [ M ] \setminus [m] $. Then $ m $ must equal $ n $ and, since $ \pi $ is non-crossing, we have that
      $
      \pi = \{ ( j, M + 1 - j): \  j \in [ m ] \}.
     $

     Denote by $ \oi $ the $ M $-tuple  $ (i_1,i_2, \dots, i_M ) $. Also, let $ \overrightarrow{j} = ( i_1, i_2, i_{ m}, i_{ m+1}, i_{ m +2})$ and denote by $ \overrightarrow{i_{\ast}} $ the $(M-6) $ -tuple obtained by eliminating
     the indices $ i_1, i_2, i_m, i_{ m + 1}$, $i_{ m + 2} $ and $ i_M $ from $ \oi $. Then, the definition of $ v(F, \pi )$ gives that
     \begin{align*}
     v(F, \pi) & = \sum_{ \oi \in [ N ]^M }
     \prod _{ k =1}^m
     \varphi \big(
     c_{ i_k i_{ k + 1}}^{ \varepsilon(k)}
     c^{\varepsilon(M + 1- k) }_{i_{ M + 1 - k }, i_{ M + 2 - k }} \big)\\
     = \sum_{ \overrightarrow{j} \in [N]^6}&
     [
     \varphi\big(
      c^{ \varepsilon(1)}_{ i_1 i_2}
     c^{\varepsilon(M)}_{i_Mi_1}
     \big)
      \varphi\big(
           c^{ \varepsilon(m)}_{ i_m i_{m+1}}
          c^{\varepsilon(m+1)}_{i_{m+1}i_{m+2}}
          \big) \cdot
          \sum_{ \overrightarrow{j}
           \in [ N]^{M-6} }
           \prod_{ k=2}^{m-1}
           \varphi \big(
                c_{ i_k i_{ k + 1}}^{ \varepsilon(k)}
                c^{\varepsilon(M + 1- k) }_{i_{ M + 1 - k }, i_{ M + 2 - k }} \big)
           ]
     \end{align*}

     On the other hand, since $ \varepsilon(m) = \varepsilon(m +1)$, we have that
     \[
     \varphi\big(
                c^{ \varepsilon(m)}_{ i_m i_{m+1}}
               c^{\varepsilon(m+1)}_{i_{m+1}i_{m+2}}
               \big) =\frac{1}{N}
                \delta_{ i_{ m}, i_{m+2}},
     \]
     hence, denoting $ \displaystyle \mathcal{E}_{\pi}(i_2, i_{m},  i_{m+2}, i_M)
     =
     \sum_{ \overrightarrow{j}
                \in [ N]^{M-6} }
                \prod_{ k=2}^{m-1}
                \varphi \big(
                     c_{ i_k i_{ k + 1}}^{ \varepsilon(k)}
                     c^{\varepsilon(M + 1- k) }_{i_{ M + 1 - k }, i_{ M + 2 - k }} \big) $,
we have that
\begin{align*}
 v(F, \pi) &= \sum_{ i_1, i_2, i_M =1}^N
 \varphi\big(
       c^{ \varepsilon(1)}_{ i_1 i_2}
      c^{\varepsilon(M)}_{i_Mi_1}
      \big)
      \cdot
 \sum_{ i_m, i_{m+1}, i_{m+2} =1}^N
  \frac{1}{N} \delta_{ i_m i_{m+2}}
 \cdot
  \mathcal{E}_{\pi}(i_2, i_{m},  i_{m+2}, i_M) \\
  &=
  \sum_{ i_1, i_2, i_M =1}^N
   \varphi\big(
         c^{ \varepsilon(1)}_{ i_1 i_2}
        c^{\varepsilon(M)}_{i_Mi_1}
        \big)
        \cdot
   \sum_{ i_m, i_{m+1}=1}^N
  \frac{1}{N}
 \cdot
   \mathcal{E}_{\pi}(i_2, i_{m},  i_{m}, i_M) \\
   &=
   N \cdot
   \sum_{ i_1, i_2, i_M =1}^N
      \varphi\big(
            c^{ \varepsilon(1)}_{ i_1 i_2}
           c^{\varepsilon(M)}_{i_Mi_1}
           \big)
           \cdot
 \mathcal{E}_{\pi}(i_2, i_{m},  i_{m}, i_M).
\end{align*}

 With the same notations,
 \[
 v(E, \pi) =
  \sum_{ \overrightarrow{j} \in [ N]^6}
  \varphi\big(
  c^{ \varepsilon(1)}_{ i_1i_2}
  c^{ \varepsilon(M)}_{ i_Mi_{m+1}}
  \big)
  \varphi\big(
  c^{\varepsilon(m)}_{i_m i_1}
  c^{\varepsilon(m+1)}_{i_{m+1}i_{m+2}}
  \big)
  \cdot
  \mathcal{E}_{\pi}(i_2, i_{m},  i_{m+2}, i_M).
 \]
 Since $ \varepsilon(m) = \varepsilon(m+1) $, we have that
     $\displaystyle \varphi\big(
       c^{\varepsilon(m)}_{i_m i_1}
       c^{\varepsilon(m+1)}_{i_{m+1}i_{m+2}}
       \big) = \frac{1}{N} \delta_{ i_1 i_{m+1}}
\delta_{i_{m} i_{ m+2}}
$, hence
\begin{align*}
v(E, \pi) = &
 \sum_{ i_1, i_2, i_M =1}^N
   \varphi\big(
         c^{ \varepsilon(1)}_{ i_1 i_2}
        c^{\varepsilon(M)}_{i_Mi_1}
        \big)
        \cdot
        \sum_{i_m, i_{m+2} =1}^N \frac{1}{N}
        \delta_{i_m, i_{m+2}} \cdot
        \mathcal{E}_\pi( i_2, i_m, i_{ m+2}, i_M)\\
        &=
  \sum_{ i_1, i_2, i_M =1}^N
     \varphi\big(
           c^{ \varepsilon(1)}_{ i_1 i_2}
          c^{\varepsilon(M)}_{i_Mi_1}
          \big)
          \cdot
          \frac{1}{N} \sum_{ i_m =1}^N \mathcal{E}_\pi( i_2, i_m, i_m, i_M) \\
          &=
           \sum_{ i_1, i_2, i_M =1}^N
                \varphi\big(
                      c^{ \varepsilon(1)}_{ i_1 i_2}
                     c^{\varepsilon(M)}_{i_Mi_1}
                     \big)
                     \cdot
                      \mathcal{E}_\pi( i_2, i_m, i_m, i_M)
\end{align*}
  hence (\ref{eq:vv}).

Suppose now that $ \varepsilon(m) \neq\varepsilon(m+1) $. We shall show first that
\begin{equation}\label{eq:ii}
v(E, \pi) = \left\{
\begin{array}{ll}
O(N^{-1})\  &\textrm{if $(k,l)\in\pi$ with $\varepsilon(k) \neq \varepsilon(l) $ and $ k, l$}\\
 & \hspace{3cm}\textrm{ in the same block of $ \gamma $}\\
O(N^{-1})\ &\textrm{if $ \pi $ has at least 3 blocks with elements} \\
& \hspace{3cm} \textrm{ from different blocks of $ \gamma$}\\
O(1) \ &\textrm{otherwise}
\end{array}
\right.
\end{equation}

 We will prove (\ref{eq:ii}) by induction on $M $. If $ M = 2 $, the first two relations are vacuously true, while the last one follows from
 \begin{align*}
 \varphi\big(\Tr(S)\Tr(S^t)\big) =
  \varphi\big( (\sum_{ i =1}^N c_{i i }) \cdot (\sum_{ j =1}^N c_{j j })\big)
   = \sum_{ i, j =1}^N \varphi(c_{ii} c_{jj})
   =
   \sum_{ i, j =1}^N \frac{1}{N} \delta_{ i j } = 1.
 \end{align*}

  For the induction step, let $ (k, k+l )$ be a block of $\pi $. Then the restriction of $ \pi $ to $ [k+l] \setminus [k -1] $ is noncrossing, hence it contains a block of the form $ (p,  p+1) $.
   If $k $ and $ k + l $ are from the same block of $ \gamma $,
   then so are $p $ and $ p+ 1 $,
   therefore applying the first part of Remark \ref{rem:v},
   we get that $ v( E, \pi ) = v ( E_{(p)}, \pi_{ (p)}) $ and, since $ \pi _{ (p)}  $ is in $ NC_2(M, \beta, c) \setminus NC_2(M, \gamma, s) $, the conclusion follows from the induction hypothesis.
  If $ \pi $ does not connect elements from the same block of $ \gamma $, then $ p $ and $ p + 1 $ are from different blocks of $ \gamma $, so $ \varepsilon(p) \neq \varepsilon(p + 1) $, but from the same block of $ \beta $. The second part of Remark \ref{rem:v} gives then $ v( E, \pi ) = \frac{1}{N}v ( E_{(p)}, \pi_{ (p)}) $ and the conclusion follows again from the induction hypothesis.

 We can suppose then that $ \pi $ contains only pairs from different cycles of $ \beta$, that is $ \pi = \{ (j, M +1 - j ): \ j \in [ m ] \} $. Denoting now $ \oi $ the multi-index $(i_1, i_{-1}, \dots, i_{M}, i_{- M}) $, with the notations from Section \ref{section:highcu} and according ro Remark \ref{rem:cycle}, we have that
  \[
       v(E, \pi)
       =
                 N^{ \sharp ( \widetilde{\gamma} \vee \epsilon \widetilde{\pi} ) - \frac{M}{2} }.
      \]
      As shown in the proof of Theorem \ref{thm:highc}, each cycle of
       $ \widetilde{\gamma} \vee \epsilon \widetilde{\pi} $
        has at least 4 elements, hence $ v(E, \pi) \leq 1 $ for all $ \pi $ as above. It only remains to show that
        $ \widetilde{\gamma} \vee \epsilon \widetilde{\pi} $
        contains a cycle with strictly more than 4 elements if $ \pi $ has at least 3 blocks with elements from different blocks of $\gamma $. In this case $ m \geq 3 $, so  $(1, M)$, $ ( m-1, m+2) $ and $ (m, m+1) $ are distinct blocks of $ \pi $ (see the figure below).

\begin{center}
\includegraphics[height=23mm]{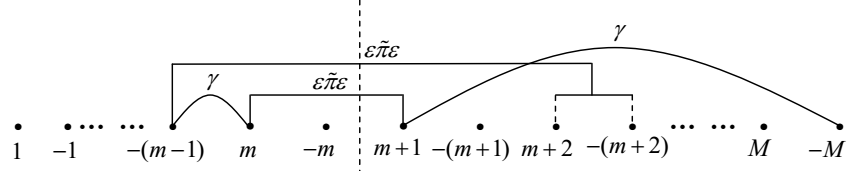}
\end{center}

         If $ \varepsilon(m) = (1) $ and $ \varepsilon(m +1) = (-1) $, we have that
          \begin{align*}
         \epsilon \widetilde{\pi} \epsilon (m) = \epsilon \widetilde{\pi} ( m)
         = \epsilon( - (m+1) ) = m+1.
          \end{align*}
          Similarly, if $ \varepsilon(m) = (-1) $ and $ \varepsilon(m +1) = (1) $, we have that
       \begin{align*}
                \epsilon \widetilde{\pi} \epsilon (  m)
                 = \epsilon \widetilde{\pi} ( - m)
                = \epsilon( -( m+1 ) ) =  m+1,
                 \end{align*}
so $ m $ and $ m+1 $ are in the same block of
 $ \epsilon \widetilde{\pi} \epsilon $.

 On the other hand,   $ \widetilde{\gamma} (- ( m -1) ) = m $ and $  \widetilde{\gamma} (m + 1) = -M $, while
  \[
  | \epsilon \widetilde{\pi} \epsilon ( - ( m -1)) | = | \pi ( m -1)| = m +2. \]

  Therefore $ -(m-1), m, m +1, -M $ and $ \epsilon \widetilde{\pi} \epsilon ( - ( m -1)) $ are 5 distinct elements of the same cycle of $  \widetilde{\gamma} \vee \epsilon \widetilde{\pi} $, hence the proof of (\ref{eq:vv}) is complete.

  Last, we shall show that property (\ref{eq:ii}) implies Theorem \ref{thm:2c}(ii). Suppose though that $ \omega(s) \neq \omega(s+1) $ and $ s\neq r $ or $ s,r \geq 3 $.

  Let $ \pi \in NC_2( M, \gamma, c) \setminus NC_2( M, \beta, s) $. It suffices to show that $\displaystyle \lim_{ N \longrightarrow \infty } v(E, \pi) = 0 $.
  Since $ \pi \notin NC_2(M, \beta, c) $, each block of $ \beta $ must be connected to at least one other block of $\beta $.

   If $ \pi $ connects two blocks of $ \beta $ from the same block of $ \gamma $, then, since it si non-crossing, $ \pi $ also connects two consecutive blocks $ B_k, B_{ k +1} $ from the same block of $ \gamma $. But $ \omega(k) \neq \omega(k+1) $ and the first part of property (\ref{eq:ii}) implies that
  $ v(E, \pi) = O(N^{-1})$.

  Suppose that $ \pi $ connects only blocks of $ \beta $ from different blocks of $\gamma $.
  If $ s, r \geq 3 $, then $ v(E, \pi) = O(N^{-1}) $ from the second part of property (\ref{eq:ii}). If $ s =1 $ and $ r = 2 $, from the first part of Remark \ref{rem:61}, we can suppose that $\beta $ has 3 cycles, first with 2 elements and the others with one element.

   Then
 \begin{align*}
  v(E, \pi) = \sum_{ \oi \in [ N]^4}
  \varphi\big( c^{ \omega(1)}_{ i_1 i_2} c^{\omega(3)}_{i_4i_4}\big)
  \varphi\big( c^{\omega(1)}_{i_2i_1} c^{\omega(3)}_{i_3i_3}\big)
  =
  \sum_{ i_1, i_2, i_3, i_4 = 1}^N
  \frac{1}{N} \delta_{i_1i_4}\delta{i_2i_4}\delta_{i_1i_3}\delta_{i_2 i_3}= \frac{1}{N}.
 \end{align*}
  The case $ s =2$ and $ r = 1 $ is similar.

  \end{proof}

\begin{remark}
 \emph{Since  $ S_N^{(1)} = S_N $ and $ S_N^{(-1)} =S_N^t $ are asymptotically free, if $ \omega(s) \neq \omega(s+1) $ we have that}
 \[\lim_{N\rightarrow\infty}
        \phi(\alpha^{(N)}_s\alpha^{(N)}_{s+1}) =0
 \]
\emph{so the formula from (i) also holds true under the assumptions from (ii)}.
\end{remark}


\bibliographystyle{alpha}


\end{document}